\documentclass[10pt, reqno]{amsart}

\usepackage{amsmath,amssymb,amsthm,amsfonts,verbatim}
\usepackage{microtype}
\usepackage[all,2cell]{xy}
\usepackage{mathtools}
\usepackage{graphicx}
\usepackage{pinlabel}
\usepackage{hyperref}
\usepackage{mathrsfs}
\usepackage{color}
\usepackage[dvipsnames]{xcolor}
\usepackage{enumerate}
\usepackage{cite}
\usepackage{soul}

\CompileMatrices

\usepackage[top=1.2in,bottom=1.2in,left=1in,right=1in]{geometry}

\usepackage{hyperref} 
\hypersetup{          
	colorlinks=true, breaklinks, linkcolor=[RGB]{51 102 204}, filecolor=Orchid, urlcolor=[RGB]{51 102 204},
	citecolor=Orchid, linktoc=all, }
\usepackage[nameinlink]{cleveref}

\theoremstyle{plain}
\newtheorem{theorem}{Theorem}[section]
\newtheorem{maintheorem}{Theorem}

\newtheorem{proposition}[theorem]{Proposition}
\newtheorem{lemma}[theorem]{Lemma}

\newtheorem{corollary}[theorem]{Corollary}

\theoremstyle{definition}
\newtheorem{definition}[theorem]{Definition}
\newtheorem{example}[theorem]{Example}
\newtheorem{construction}[theorem]{Construction}

\newtheorem{remark}[theorem]{Remark}

\newtheorem{convention}[theorem]{Convention}

\newcommand{\nc}{\newcommand}
\nc{\dmo}{\DeclareMathOperator}

\nc{\Q}{\mathbb{Q}}
\nc{\F}{\mathbb{F}}
\nc{\R}{\mathbb{R}}
\nc{\Z}{\mathbb{Z}}
\nc{\C}{\mathbb{C}}
\nc{\N}{\mathbb{N}}
\nc{\Ell}{\mathcal{L}}
\nc{\M}{\mathcal{M}}
\nc{\K}{\mathcal{K}}
\nc{\I}{\mathcal{I}}
\nc{\T}{\mathcal T}
\nc{\U}{\mathcal U}
\nc{\disk}{\mathbb{D}}
\nc{\hyp}{\mathbb{H}}
\renewcommand{\P}{\mathbb{P}}
\nc{\CP}{\mathbb{CP}}
\nc{\RP}{\mathbb{RP}}
\nc{\cS}{\mathcal{S}}
\nc{\inv}{^{-1}}
\nc{\s}{\sigma}
\dmo{\Mod}{Mod}
\dmo{\PMod}{PMod}
\dmo{\LMod}{LMod}
\dmo{\Diff}{Diff}
\dmo{\Homeo}{Homeo}
\dmo{\dist}{dist}
\dmo\BDiff{BDiff}
\dmo\SO{SO}
\dmo\Hom{Hom}
\dmo\SL{SL}
\dmo\rank{rank}
\dmo\sig{sig}
\dmo\Out{Out}
\dmo\Aut{Aut}
\dmo\Inn{Inn}
\dmo\GL{GL}
\dmo\PGL{PGL}
\dmo\Gr{Gr}
\dmo\PSL{PSL}
\dmo\EU{EU}
\dmo\BHomeo{BHomeo}
\dmo\EHomeo{EHomeo}
\dmo\EDiff{EDiff}
\dmo\Disc{Disc}
\dmo\Aff{Aff}

\nc{\pair}[1]{\ensuremath{\left\langle #1 \right\rangle}}
\nc{\abs}[1]{\ensuremath{\left| #1 \right|}}
\nc{\abcd}[4]{\ensuremath{\left(\begin{array}{cc} #1 & #2 \\ #3 & #4 \end{array}\right)}}
\nc{\into}\hookrightarrow
\dmo{\Isom}{Isom}
\nc{\normal}{\vartriangleleft}
\dmo{\Vol}{Vol}
\dmo{\im}{Im}
\dmo{\Push}{Push}
\dmo{\Conf}{Conf}
\dmo{\PConf}{PConf}
\dmo{\id}{id}
\dmo{\Jac}{Jac}
\dmo{\Pic}{Pic}
\dmo{\Stab}{Stab}
\dmo{\Arf}{Arf}
\dmo{\End}{End}
\dmo{\Gal}{Gal}
\dmo{\lcm}{lcm}
\dmo{\ab}{ab}
\dmo{\opp}{op}
\dmo{\SU}{SU}
\dmo{\OT}{\Omega \mathcal{T}}
\dmo{\OM}{\Omega \mathcal{M}}
\dmo{\spin}{spin}
\dmo{\even}{even}
\dmo{\odd}{odd}
\dmo{\comp}{\mathcal{H}}
\dmo{\Mgk}{\mathcal{M}_{g, \underline{\kappa}}}
\dmo{\orb}{orb}
\dmo{\AJ}{AJ}
\dmo{\Ck}{\mathsf{C}(\underline{\kappa})}
\dmo{\Int}{Int}
\dmo{\pr}{pr}
\dmo{\lab}{lab}
\dmo{\Sym}{Sym}
\dmo{\Ann}{Ann}
\dmo{\Rad}{Rad}
\dmo{\Ind}{Ind}
\DeclareRobustCommand{\subsetsof}{\genfrac [ ]{0pt}{}}

\nc{\Span}[1]{\operatorname{Span}(#1)}

\renewcommand{\epsilon}{\varepsilon}
\renewcommand{\tilde}{\widetilde}
\renewcommand{\le}{\leqslant}
\nc{\coloneq}{\mathrel{\mathop:}\mkern-1.2mu=}
\nc{\margin}[1]{\marginpar{\scriptsize #1}}
\nc{\para}[1]{\medskip\noindent\textbf{#1.}}
\definecolor{myblue}{RGB}{102,153, 255}
\definecolor{myred}{RGB}{204,0,0}
\definecolor{mygreen}{RGB}{0,204,0}
\definecolor{myorange}{RGB}{255,102,0}
\definecolor{mypurple}{RGB}{138,43,226}
\nc{\red}[1]{\textcolor{myred}{#1}}
\nc{\blue}[1]{\textcolor{myblue}{#1}}
\nc{\proofof}[1]{\noindent {\em Proof (of #1).}}

\nc{\SB}[1]{{\bf SB}$(#1)$}
\nc{\SC}[1]{{\bf SC}$(#1)$}
\nc{\cSC}[1]{{\bf cSC}$(#1)$}
\nc{\AB}[1]{{\bf AB}$(#1)$}
\nc{\AC}[1]{{\bf AC}$(#1)$}
\nc{\IH}[1]{{\bf IH}$(#1)$}

\title{Totally symmetric sets in the general linear group}

\author{Noah Caplinger and Nick Salter}
\email{NC: ncaplinger@gatech.edu}
\email{NS: nsalter@nd.edu}
\thanks{NC is supported by the President's Undergraduate Research Award at Georgia Tech. NS is supported by NSF Award No. DMS-2153879.}
\address{NC: Department of Mathematics, Georgia Institute of Technology, 686 Cherry St NW, Atlanta, GA 30332}
\address{NS: Department of Mathematics, University of Notre Dame, Hurley Hall, Notre Dame, IN 46556}
\date{April 25, 2022}

\begin{document}
\maketitle

\begin{abstract}
A totally symmetric set is a finite subset of a group for which any permutation of the elements can be realized by conjugation in the ambient group. Such sets are rigid under homomorphisms, and so exert a great deal of control over the algebraic structure. In this paper we introduce a more general perspective on total symmetry, and formulate a notion of ``irreducibility'' for totally symmetric sets in the general linear group. We classify irreducible totally symmetric sets, as well as those of maximal cardinality.
\end{abstract}

\section{Introduction}
Let $G$ be a group. The notion of a {\em totally symmetric set} in $G$ was introduced by Kordek--Margalit \cite{Kordek-Margalit} as an axiomatization of a structure that had previously appeared in a variety of contexts in geometric group theory. A totally symmetric set $\mathcal A \subset G$ is a finite set of $k$ elements, {\em such that any permutation of the elements of $\mathcal A$ can be obtained by conjugation by $G$} (see \Cref{definition:totalsym} and \Cref{example:commtss}). Totally symmetric sets are tightly controlled under group homomorphisms (\Cref{remark:hom}) and have been applied to better understand rigidity properties of many of the central groups in contemporary geometric group theory, e.g. braid groups and mapping class groups; see below. In these applications, one typically further insists that the elements of $\mathcal A$ pairwise-commute; $\mathcal A$ then acts as an analogue of a maximal torus in a compact Lie group.

In this paper we study and classify totally symmetric sets in the general linear group $\GL_n(\C)$---in fact, it is more natural to consider totally symmetric sets of {\em endomorphisms}. Our first main result shows that large totally symmetric sets do not act on low-dimensional spaces.

\begin{maintheorem}\label{theorem:sizebound}
Any totally symmetric set $\mathcal A \subset \End(\C^n)$ has cardinality $k \le n+1$. If $\mathcal A$ is moreover commutative, then $k \le n$.  
\end{maintheorem}

 In both the commutative and non-commutative setting, the totally symmetric sets of maximal cardinality can be classified completely. We note that these results give a sharp answer to a question posed by the first author in \cite{Caplinger-Kordek}.

\begin{maintheorem}
\label{maintheorem:max}
Let $\mathcal{A}\subset \End(\C^n)$ be a totally symmetric set (commutative or otherwise) of the maximal cardinality allowed by \Cref{theorem:sizebound}. Then
\begin{enumerate}
    \item If $\mathcal A$ is noncommutative and $n \ne 5$, then it arises via the ``noncommutative simplex construction'' of \Cref{example:ncsimplex}. 
    \item If $\mathcal A$ is noncommutative and $n = 5$, then it arises either via the noncommutative simplex construction or else is the ``$\tilde \Sigma_5$ construction'' of \Cref{s5construction}.
    \item If $\mathcal A$ is commutative and $n \ne 4$, then $\mathcal{A}$ is either the ``standard construction'' of \Cref{example:standard} or the ``simplex construction'' of \Cref{example:simplex}.
    \item If $\mathcal A$ is commutative and $n = 4$, then $\mathcal{A}$ is either standard, simplex, or the sporadic construction of \Cref{lemma:4}.
\end{enumerate}
\end{maintheorem}

Our main tool is the formulation of a notion of {\em irreducibility} of a totally symmetric set in $\End(\C^n)$ (\Cref{definition:irreducible}). We find that irreducible commutative totally symmetric sets are very rigid, and admit a classification, like the irreducible representations of the symmetric group, essentially in terms of integer partitions.

To formulate this classification, we define a {\em weight} as a function $\vec \lambda: [k] \to \C$ (here and throughout, $[k]$ denotes a set of $k$ elements). In \Cref{example:partition}, we show how to use $\vec \lambda$ to construct a totally symmetric set $\mathcal A(\vec \lambda)$ acting as endomorphisms of vector space $V_{\vec\lambda}$. A brief description of $V_{\vec \lambda}$ is as the space spanned by functions of the form $\lambda \circ \sigma$ for $\sigma$ in the symmetric group $\Sigma_k$ (we treat two such functions as identical if they agree everywhere). The level sets of $\vec \lambda$ partition $[k]$, say with parts of size $k_1, \dots, k_m$; basic combinatorics computes $\dim(V_{\vec \lambda})$ as the multinomial coefficient $\binom{k}{k_1, \dots, k_m}$. The element $A_i$ acts diagonally on $V_{\vec \lambda}$, taking the eigenvalue $\lambda(\sigma(i))$ on the basis element $\lambda \circ \sigma$, and $\Sigma_k$ acts on $V_{\vec \lambda}$ by precomposition, yielding total symmetry. We find that every irreducible commutative totally symmetric set has this form.

\begin{maintheorem}
\label{mainthm:summary}
Let $\mathcal A \subset \End(\C^n)$ be an irreducible commutative totally symmetric set of cardinality $k$. Then there is a weight $\vec \lambda$ for which 
\[
\mathcal A \cong \mathcal A (\vec \lambda).
\]
The dimension of the space $V_{\vec \lambda} \cong \C^n$ on which $\mathcal A(\vec \lambda)$ acts is given by the multinomial coefficient $\binom{k}{k_1,\dots, k_m}$.
\end{maintheorem}
See \Cref{definition:iso} for the meaning of an isomorphism of totally symmetric sets. 

As is already suggested by the terminology of ``weights'' and ``irreducibility'' for totally symmetric sets (as well as by the role played by integer partitions in the classification), a central theme of the paper is the exploration of how concepts from representation theory have analogues in this setting. The core of \Cref{mainthm:summary} is established in \Cref{mainthm:irred}, where we show that every irreducible commutative totally symmetric set has the structure of an {\em induction} (defined in \Cref{induction}) which behaves similarly to its representation-theoretic namesake. 

One important point of departure for these theories is the failure of {\em semisimplicity}---where a representation of a finite group decomposes as a direct sum of irreducibles, a reducible totally symmetric set can be built as a nontrivial extension of lower-dimensional totally symmetric sets (i.e. with nontrivial Jordan blocks). At present we do not have a complete analysis of the ``extension problem'' for reducible totally symmetric sets -- \Cref{lemma:4} hints at the difficulty of this.

It is interesting to note that our analysis of commutative totally symmetric sets requires us to understand non-commutative sets as well (\Cref{remark:noncom}). Historically, commutative totally symmetric sets were the first to be defined and investigated, but there are contexts in which non-commutative totally symmetric sets are an important tool in their own right. For instance, in \Cref{proposition:sn}, we use \Cref{maintheorem:max} to give a conceptual understanding of why the symmetric group does not admit low-dimensional non-abelian representations:\\

\noindent \textbf{\Cref{proposition:sn}} {\em For $n \ne 4$, the symmetric group $\Sigma_n$ admits no non-abelian representations over $\C$ of dimension $d < n-1$.}\\

Of course this fact is well-known, but the usual proof requires the full apparatus of the partition-irrep correspondence for representations of $\Sigma_n$. Farb has asked for a conceptual understanding of this, and to that end, our proof can be simply summarized as follows: {\em $\Sigma_n$ contains the $(n-1)$-element totally symmetric set $\{(1,i) \mid 2 \le i \le n\}$, which obstructs the existence of representations in dimension below $n-1$.}

\para{History and context} The basic principle underlying the utility of totally symmetric sets is the so-called {\em persistence lemma} (\Cref{lemma:rigid}; see also \Cref{remark:hom}), which asserts that if $\mathcal A \subset G$ is a totally symmetric set and $f: G \to H$ is a homomorphism, then $f(\mathcal A)$ is either a totally symmetric set of the same cardinality or else a singleton (i.e. a {\em degenerate} totally symmetric set, in the terminology of \Cref{definition:totalsym}). Thus classifying totally symmetric sets in a fixed group $G$ goes a long way towards understanding homomorphisms both in and out of $G$.

Totally symmetric sets are of particular use in the study of braid groups, where a map $f: B_n \to G$ has cyclic image if and only if the image of the totally symmetric set $X_n = \{\sigma_{2i-1}\}_{i = 1}^{\lfloor n/2 \rfloor}$ under $f$ is degenerate. This was the basic premise of \cite{Finite_Quotients}, \cite{Caplinger-Kordek} and \cite{Nancy-Yvon}, which gave superexponential bounds on the size of non-cyclic quotients of braid groups. These were established by proving that groups must be suitably large to contain a totally symmetric set of a certain size. If $G$ is too small, it cannot contain a totally symmetric set of cardinality $|X_n|$, forcing $f(X_n)$ to be degenerate for any $f:B_n \to G$. In \cite{UpperBoundsTSS}, Kordek-Li-Partin show that this is a somewhat generic phenomenon: large totally symmetric sets in finite groups are quite rare.

In \cite{Chen-Mukherjea} and \cite{Kordek-Margalit}, Chen-Mukherjea and Kordek-Margalit classify homomorphisms $B_n \to \Mod{S_g}$ and $[B_n,B_n] \to B_n$ respectively. These papers go further by classifying totally symmetric sets in their respective codomains and using this information to deduce the possible images of totally symmetric sets in the domain. Both Kordek-Margalit and Chen-Mukherjea use an auxiliary construction, {\em totally symmetric multicurves} to study totally symmetric sets in mapping class groups. We will similarly  make use of {\em totally symmetric arrangements of subspaces} to study totally symmetric sets in $\End(V)$, and will unify these various notions of total symmetry in \Cref{subsection:totalsym}.

The common theme of these papers is that
homomorphisms $f:G \to H$ can be studied by analyzing the totally symmetric sets of $G$ and $H$. Furthermore, both bounding the size of totally symmetric sets, and classifying them outright can give information about the corresponding homomorphisms. Our results fit neatly into this program, and we anticipate that they may be used to further understand representations of groups with totally symmetric sets.

\para{Representation theory of braid and mapping class groups} There is an analogy between the notion of a commutative totally symmetric set and that of a maximal torus in Lie theory---the total symmetry property is reminiscent of the action of the Weyl group. This raises the possibility of studying the representation theory of groups $G$ with large totally symmetric sets (e.g. braid and mapping class groups) in terms of {\em weights}, i.e. in terms of the totally symmetric set in $\GL_n(\C)$ obtained as the image of a chosen maximal totally symmetric set for $G$. Further progress in this direction will hinge on being able to analyze the extent to which the entire representation is determined by the weights, e.g. by incorporating more of the algebraic structure of $G$, such as the presence of braid relations in the case of braid and mapping class groups.

\para{Organization} In \Cref{section:basic}, we introduce the basics of our formulation of total symmetry. We also recall the facts from linear algebra that underlie the theory of totally symmetric sets in $\GL_n(\C)$.

The remainder of the paper is divided into two parts---the first half is concerned with a classification of irreducible commutative totally symmetric sets, culminating in the proof of \Cref{mainthm:summary} as \Cref{corollary:partition} of \Cref{mainthm:irred}, and the second half is concerned with classifying totally symmetric sets (commutative and otherwise) of maximal cardinality, as presented in \Cref{theorem:sizebound} and \Cref{maintheorem:max}.

In \Cref{section:irred}, we introduce the notion of {\em irreducibility} of a totally symmetric set in $\GL_n(\C)$ and discuss some constructions arising from this. In \Cref{section:induction}, we discuss the {\em induction construction} that lies at the heart of our analysis of irreducible totally symmetric sets, and illustrate this with some examples. The proof of \Cref{mainthm:summary} is carried out in \Cref{section:proofA}; this relies on the notion of {\em depth}, which we discuss first.

The second half begins in \Cref{section:arrangement} with a discussion of totally symmetric arrangements of vector subspaces, which is the basic tool by which we study both reducible and non-commuting totally symmetric sets in $\End(\C^n)$. \Cref{theorem:sizebound} and \Cref{maintheorem:max} are then proved simultaneously by induction, along with a third statement (\Cref{theorem:tsabound}) that bounds the size of a totally symmetric subspace arrangement. These results are closely intertwined: there are procedures for constructing a totally symmetric set from a totally symmetric arrangement and vice versa, necessitating an induction that treats both at once. We outline our inductive hypotheses and establish the base cases in \Cref{section:indstart}. Sections \Cref{section:sb,section:sc,section:csc,section:abac} then treat the various inductive hypotheses. Finally in \Cref{section:reptheory}, we apply our results to the representation theory of $\Sigma_n$, proving \Cref{proposition:sn}. 

In \Cref{appendix:s5}, we give a detailed study of the structure of the special $\tilde \Sigma_5$ arrangement appearing in the text.

\para{Acknowledgements} The first author would like to thank Dan Minahan and Wade Bloomquist for their interest in the project and helpful conversations. Both authors would like to acknowledge Benson Farb and Dan Margalit for interest in the project and for comments on a preliminary draft.

\section{Basic notions}\label{section:basic}
In \Cref{subsection:totalsym}, we present a notion of ``total symmetry'' in the context of arbitrary group actions. In \Cref{subsection:gln} we specialize to the case of interest, totally symmetric sets in $\GL_n(\C)$, and establish some fundamental properties in this context.

\subsection{Total symmetry}\label{subsection:totalsym}

Totally symmetric sets were introduced by Kordek--Margalit in \cite{Kordek-Margalit}. Already in this paper, other ``total symmetry'' properties for sets in more general $G$-sets (not just groups acting on themselves by conjugation) appeared. The formulation of ``total symmetry'' we offer here allows us to unify these properties as instances of a single definition.

We must establish some notation. Given a finite subset $S$, we write $\Sym(S)$ for the group of permutations of $S$. We write $[n]$ for the set $\{1, \dots, n\}$; 
in this case we write $\Sigma_n$ in place of $\Sym([n])$.

\begin{definition}[Total symmetry]
\label{definition:totalsym}
Let $G$ be a group acting on a set $X$, given by a homomorphism $\alpha: G \to \Sym(X)$. A subset $S \subset X$ is {\em totally symmetric} if $S$ is endowed with a transitive $\Sigma_k$-action
\[
\beta: \Sigma_k \to \Sym(S)
\]
for some $k$, and a {\em set map} (see \Cref{remark:setmap})
\[
\rho: \Sigma_k \to G
\]
called the {\em realization map} such that $\alpha (\rho(\sigma)) \in \Sym(X)$ preserves $S$ for all $\sigma \in \Sigma_k$ and for which, as automorphisms of $S$,
\begin{equation}\label{equation:abr}
\alpha(\rho (\sigma)) = \beta(\sigma).
\end{equation}

A totally symmetric set $S$ is {\em degenerate} if $S$ is a singleton.
\end{definition}

\begin{remark}[Some general comments]\label{remark:general}\ 
\begin{enumerate}
    \item Colloquially, a totally symmetric set $S$ is one for which any {\em internal} symmetry of $S$ (as mediated by $\beta$) can be realized {\em externally} by the ``ambient'' group action of $G$ on $X$ under the realization map $\rho: \Sigma_n \to G$. 
    \item The realization map is thought of as {\em secondary data} attached to a totally symmetric set; in practice there will frequently be many choices for realization maps which all realize \eqref{equation:abr}, and we do not wish to privilege any particular one.
    \item There is no theoretical reason to restrict $S$ to carrying an action of $\Sigma_n$ as opposed to some more general group, but the examples in the literature are all of this form and we do not wish to needlessly overburden the notation.
\end{enumerate}
\end{remark}

\begin{example}[(Commuting) totally symmetric sets in groups]
\label{example:commtss}
The primary example we will be concerned with in this paper is as follows: $G$ will be a group, and the $G$-set $X$ will be $G$ acting on itself by conjugation. A totally symmetric set in this context is a set
\[
\mathcal A = \{A_1,\dots,A_k\} \subset G
\]
such that any permutation $\sigma \in \Sym(\mathcal A) \cong \Sigma_k$ can be realized by conjugation by some element $\rho(\sigma) \in G$.

Frequently one adds the hypothesis that the elements of $\mathcal A$ pairwise commute (in the older literature, this is part of the definition of a ``totally symmetric set''). Under this additional hypothesis, we say that $\mathcal A$ is a {\em commutative totally symmetric set}. 
\end{example}


\begin{example}
We will see that an understanding of commuting totally symmetric sets in $G = \GL(V)$ actually requires an understanding of (possibly non-commuting) totally symmetric sets in the endomorphism algebra $\End(V)$; the action of $\GL(V)$ on itself by conjugation extends to an action on $\End(V)$. Thinking about the associated eigenspaces will lead us to the other main class of totally symmetric set studied in the paper: totally symmetric arrangements of linear subspaces. This is a collection of subspaces $\mathcal{W} = \{W_1,\ldots, W_k\}$ of $V$ so that for every every $\sigma \in \Sigma_k$, there is a $\rho(\sigma) \in \GL{V}$ satisfying $\rho(\sigma)\cdot W_i = W_{\sigma(i)}$.

\end{example}


\begin{example}
A central example that we will {\em not} explore in this paper arises in the theory of mapping class groups. Let $\Sigma$ be a surface and $\Mod(\Sigma)$ the associated mapping class group. $\Mod(\Sigma)$ acts on the set $X$ of isotopy classes of simple closed curves in $\Sigma$; a {\em totally symmetric set of curves} is a totally symmetric set $S \{c_1,\ldots, c_k\} \subset X$ with respect to this action. That is, any permutation of $S$ can be realized by an element of $\Mod(S)$. Totally symmetric sets of curves play a crucial role in understanding totally symmetric sets in the mapping class group itself; any such set of curves can be promoted to a totally symmetric set of elements of $\Mod(\Sigma)$ by assigning each curve to its Dehn twist. This is philosophically similar to the relationship between linear endomorphisms and their eigenspaces.
\end{example}

\begin{remark}[Why set maps?]
\label{remark:setmap}
There is an apparently unusual feature of \Cref{definition:totalsym}: while the group actions $\alpha, \beta$ are required to be maps of groups, the realization map $\rho$ is only required to be a map of sets. We formulate the definition the way that we do because $\rho$ admits a natural description as a {\em section} of a certain group surjection, and such sections are {\em a priori} only maps of sets. To see this, we let $\Stab(S) \le G$ denote the stabilizer of $S$ under the action $\alpha$, and consider the following diagram:
\[
\xymatrix{
                & \Stab(S) \ar[d]^\alpha\\
\Sigma_k \ar@{.>}[ur]^{\rho} \ar[r]_-\beta    & \Sym(S).    
}
\]
The total symmetry condition of \Cref{definition:totalsym} is seen to be equivalent to the existence of a section $\rho$ of $\alpha$ over the subgroup $\beta(\Sigma_k) \le \Sym(S)$.

Moreover, in many of the motivating examples, $\rho$ in fact {\em cannot} be promoted to a group homomorphism. An obvious obstruction for this is if $G$ is torsion-free, e.g. in the case of $G = B_n$ the braid group, the context from which the theory of total symmetry arose.
\end{remark}

Totally symmetric sets obey a fundamental {\em rigidity} property under maps of the ambient $G$-sets, as described in \Cref{lemma:rigid} below. To formulate this, for $i = 1,2$, let $\alpha_i: G_i \to \Sym(X_i)$ give the sets $X_i$ the structure of $G_i$-sets, and let $\phi: G_1 \to G_2$ be a homomorphism. We say that $f: X_1 \to X_2$ is {\em $\phi$-intertwining} if for all $g \in G_1$,
\[
f \circ \alpha_1(g) = \alpha_2(\phi(g)) \circ f.
\]
In the case $G_1 = G_2 = G$ and $\phi = \id$ we simply say that $f$ is a {\em map of $G$-sets}. With this formulated, we can state the rigidity property for totally symmetric sets: they {\em persist} under intertwining maps of the ambient $G$-sets.
\begin{lemma}[Persistence]\label{lemma:rigid}
For $i = 1,2$, let $X_i$ be a $G_i$-set, and let $\phi: G_1 \to G_2$ be a homomorphism. Let $S_1 \subset X_1$ be totally symmetric, and let $f: X_1 \to X_2$ be $\phi$-intertwining. Then $S_2:= f(S_1)$ is totally symmetric, and $f$ restricts to a map $f: S_1 \to S_2$ of $\Sigma_n$-sets.
\end{lemma}
\begin{proof}
Let $\alpha_i: G_i \to \Sym(X_i)$ and $\beta_1: \Sigma_n \to \Sym(S_1)$ be the structure maps for the group actions; let $\rho_1: \Sigma_n \to G_1$ be the realization map for the totally symmetric set $S_1$. We must define a group action $\beta_2: \Sigma_n \to \Sym(S_2)$ and a realization map $\rho_2: \Sigma_n \to G_2$.

For $y \in S_2$ and $\sigma \in \Sigma_n$, we define 
\[
\beta_2(\sigma)(y) = f(\beta_1(\sigma)(x)),
\]
where $x \in S_1$ satisfies $f(x) = y$. We must check that this is well-defined. By total symmetry,
\[
\beta_1(\sigma)(x) = \alpha_1(\rho_1(\sigma))(x),
\]
so that
\[
f(\beta_1(\sigma)(x)) = f(\alpha_1(\rho_1(\sigma))(x)) = \alpha_2(\phi(\rho_1(\sigma)))(f(x)) = \alpha_2(\phi(\rho_1(\sigma)))(y).
\]
This visibly depends only on $y \in S_2$, showing well-definedness. This also shows that taking $\rho_2 = \phi \circ \rho_1$ endows $S_2$ with the structure of a totally symmetric set; that $f$ restricts to a map of $\Sigma_n$-sets $f: S_1 \to S_2$ is also immediate.
\end{proof}

\begin{corollary}\label{corollary:permrigid}
In the setting of \Cref{lemma:rigid}, suppose further that $S_1$ is isomorphic as a $\Sigma_n$-set to $[n]$ under the permutation action of $\Sigma_n$. Then either $f$ restricts to an isomorphism $S_1 \cong S_2$ of $\Sigma_n$-sets, or else $S_2$ is degenerate (i.e. a singleton carrying the trivial $\Sigma_n$-action). 
\end{corollary}
\begin{proof}
From \Cref{lemma:rigid}, $f$ induces a surjective map of $\Sigma_n$-sets $f: S_1 \to S_2$. If $S_1 = [n]$, it is easy to see that either $f$ is an isomorphism or else $S_2$ is a singleton: if $f(x_1) = f(x_2)$ for any $x_1 \ne x_2 \in S_1$, then for any $x_3 \ne x_1, x_2 \in S_1$, let $\sigma \in \Sigma_n$ satisfy $\sigma(x_2) = x_3$ and $\sigma(x_1) = x_1$. Then
\[
f(\sigma \cdot x_2) = f(x_3), 
\]
while 
\[
\sigma \cdot f(x_2) = \sigma \cdot f(x_1) = f(\sigma \cdot x_1) = f(x_1),
\]
showing $S_2$ is the singleton $f(x_1)$. 
\end{proof}

\begin{remark}[A standing simplifying assumption]
In the remainder of the paper, unless explicitly stated otherwise, we restrict attention to the setting of \Cref{corollary:permrigid}, i.e. where $S$ is isomorphic to $[n]$ as a $\Sigma_n$-set. In other contexts (e.g. the ``totally symmetric multicurves'' appearing in \cite{Chen-Mukherjea} and \cite{Kordek-Margalit}), $\beta$ and $S$ can be more complicated, e.g. $S$ can be the set of $k$-element subsets of $[n]$, but we will not need to pursue this further here.
\end{remark}

\begin{remark}[Collision implies collapse]\label{remark:cimpc}
A consequence of \Cref{corollary:permrigid} that we will frequently use is the following: {\em if $f: S_1 \to S_2$ is a map of totally symmetric sets, and if $f(s_1) = f(s_2)$ for any distinct elements $s_1, s_2 \in S$, then $f(S_1)$ is a singleton, i.e. a degenerate totally symmetric set.}
\end{remark}

\begin{remark}\label{remark:hom}
In the setting where $X_1$ and $X_2$ are groups $G_1,G_2$ acting on themselves via conjugation, and the map $f: X_1 \to X_2$ is given by a homomorphism $\phi$, \Cref{lemma:rigid} simply asserts that the image of a totally symmetric set under a homomorphism is again totally symmetric. Combining this observation with \Cref{remark:cimpc} yields \cite[Lemma 2.1]{Kordek-Margalit}.
\end{remark}

\subsection{Totally symmetric sets of endomorphisms}\label{subsection:gln}
For the remainder of the paper, we specialize to the following setting: $V$ will denote a finite-dimensional vector space over $\C$, $\GL(V)$ will denote the associated general linear group, and $\End(V)$ will denote the algebra of $\C$-linear endomorphisms. Note that $\End(V)$ is a $\GL(V)$-set under conjugation.

Our objective in this section is to recall the basic structure theory for sets of commuting linear endomorphisms and to apply this in order to establish some basic facts about totally symmetric sets in $\End(V)$. 

\para{Generalized eigenspaces} Let $A \in \End(V)$ be an endomorphism. Recall that the {\em generalized $\lambda$ eigenspace of degree $c$} is the kernel of $(A- \lambda I)^c$; we write this $E(A)_{\lambda,c}$, or simply $E_{\lambda,c}$ when $A$ is clear from context. In the case $c = 1$, we will often drop the subscript $1$ and write simply $E_\lambda$; we will also drop ``generalized'' from the terminology.

For increasing $c$, the spaces $E_{\lambda,c}$ form an increasing filtration
\[
E_{\lambda,0} = \{0\} \leq E_{\lambda,1} \leq \dots \leq E_{\lambda,d} = E_{\lambda, d+1} = \dots
\]
that stabilizes for some $d \ge 0$. If $A$ and $B$ are conjugate, then the conjugating map induces an isomorphism $E(A)_{\lambda, c} \cong E(B)_{\lambda,c}$ for all eigenvalues $\lambda$ and degrees $c$. Clearly $A - \lambda I$ induces maps
\[
(A- \lambda I): E_{\lambda,c} \to E_{\lambda, c-1}
\]
with kernel $E_{\lambda, 1}$ for $c \ge 1$, and it is easy to see that  $(A-\lambda I)$ induces an {\em injection} (again for $c \ge 1$)
\[
(A-\lambda I): E_{\lambda, c+1}/E_{\lambda, c} \into E_{\lambda, c}/E_{\lambda, c-1}.
\]
Considering the dimensions of these spaces, we obtain the {\em Jordan inequalities}
\begin{equation}\label{equation:jordan}
    \dim(E_{\lambda,c}/E_{\lambda, c-1}) \ge \dim(E_{\lambda, c+1} / E_{\lambda, c})
    \end{equation}
    of dimensions of generalized eigenspaces, which leads to the corollary
    \begin{equation}\label{equation:jordan2}
    \dim(E_{\lambda,c}) \ge c \dim(E_{\lambda,c}/E_{\lambda,c-1}),
    \end{equation}
valid for all $c \ge 1$.

\para{Generalized eigenspaces of totally symmetric sets} Let $\mathcal A = \{A_1, \dots, A_k\} \subset \End(V)$ be a totally symmetric set. For a fixed $i \in [k]$, we write 
\[
E_{\lambda,c}^i \subset V
\]
to denote the degree-$c$ generalized $\lambda$ eigenspace of $A_i$. Since $A_i$ and $A_j$ are conjugate, the spaces $E_{\lambda,c}^i$ and $E_{\lambda,c}^j$ are isomorphic for all pairs of indices $i,j$. Thus we speak of $\lambda$ as being an eigenvalue of $\mathcal A$, not merely of an individual $A_i \in \mathcal A$. When $\mathcal A$ is commutative, each generalized eigenspace $E_{\lambda,c}^i$ is invariant under each $A_j \in \mathcal A$.\\

The $j$-fold eigenspaces introduced below will play a central role in what follows, and \Cref{lemma:conjformula}, while elementary, lies at the heart of our analysis of totally symmetric sets in $\End(V)$.

\begin{definition}[$j$-fold eigenspace]
\label{definition:pfold}
Let $\mathcal A \subset \End(V)$ be a totally symmetric set of cardinality $k$. Fix $j \le k$, and let $\lambda$ be an eigenvalue of $\mathcal A$. Let $S \subset [k]$ be a subset of cardinality $j$. We define the {\em $j$-fold generalized eigenspace of degree $c$} associated to $S \subset [k]$ as the intersection
\[
E_{\lambda,c}^S := \bigcap_{i \in S}E_{\lambda,c}^i.
\]
As above, in the case $c = 1$ we will frequently drop $c$ from the notation.
\end{definition}

For $V$ a finite-dimensional vector space, let $\Gr(V)$ denote the union of the Grassmannians $\Gr_d(V)$ for $d \ge 0$. Observe that $\Gr(V)$ is a $\GL(V)$-set. For $\lambda \in \C$ and a positive integer $c$, there is a map of $\GL(V)$-sets
\[
E(\,\cdot\,)_{\lambda,c}: \End(V) \to \Gr(V)
\]
which takes $A \in \End(V)$ to $E(A)_{\lambda,c} \in \Gr(V)$. (We remark that when $\End(V)$ and $\Gr(V)$ are endowed with their usual topologies, this is of course discontinuous, but this will be irrelevant for our purposes). \Cref{lemma:conjformula}, while nothing more than a formulation of a basic linear-algebraic principle in our language, will be essential in what follows.

\begin{lemma}
\label{lemma:conjformula}
Let $\mathcal A = \{A_1, \dots, A_k\} \subset \End(V)$ be a totally symmetric set with realization map $\rho: \Sigma_k \to \GL(V)$. Let $\lambda$ be an eigenvalue of $\mathcal A$, let $S \subset [k]$ be a subset, and let $\sigma \in \Sigma_k$ be arbitrary. Then $\rho(\sigma)$ induces isomorphisms
\[
\rho(\sigma): E_{\lambda,c}^{S} \to E_{\lambda,c}^{\sigma(S)}
\]
for all subsets $S \subset [k]$ and all degrees $c \ge 0$. 
\end{lemma}

\begin{remark}
In other words, $j$-fold eigenspaces form a totally symmetric set (a ``totally symmetric arrangement'', in the language of \Cref{section:arrangement}), where the ambient action is $\GL(V)$ acting on some $\Gr_d(V)$ and the $\Sigma_k$-action is on $k$-element subsets of $[n]$.
\end{remark}



\begin{remark}
The reader familiar with the theory of totally symmetric multicurves as in \cite{Kordek-Margalit} and \cite{Chen-Mukherjea} may note the structural similarity between this and the notion of $j$-fold eigenspaces.
\end{remark}

\para{Isomorphism} In the course of our analysis, we will have occasion to consider totally symmetric sets acting on isomorphic vector spaces, necessitating a notion of isomorphism of totally symmetric sets. 
\begin{definition}[Isomorphism of totally symmetric sets]
\label{definition:iso}
Let $\mathcal A = \{A_1, \dots A_k\} \subset \End(V)$ and $\mathcal B = \{B_1, \dots B_k\} \subset \End(W)$ be totally symmetric sets. An isomorphism $\mathcal A \cong \mathcal B$ of totally symmetric sets is an isomorphism $T: V \to W$ such that the equation
\[
B_i = T A_i T^{-1}
\]
holds for all $i \in [k]$.
\end{definition}
\begin{remark}
If $T$ is an isomorphism of totally symmetric sets, then conjugation by $T$ will induce a bijection between the set of realization maps for $\mathcal A$ and the set of those for $\mathcal B$. In keeping with the principle that the realization map is of secondary importance, we do not impose any requirements that $T$ conjugate a {\em specific} realization map for $\mathcal A$ into one for $\mathcal B$.
\end{remark}

\section{Irreducibility}\label{section:irred}

We come now to the central new definition in our analysis of totally symmetric sets in $\End(V)$, that of {\em (ir)reducibility} and the corresponding {\em restriction} and {\em quotient} constructions (\Cref{definition:irreducible}, \Cref{construction:restriction}). As the name suggests, irreducibility is motivated by the analogous concept in representation theory. \Cref{mainthm:irred} will demonstrate that the irreducible totally symmetric sets of a fixed cardinality are extremely rigid. Unlike in the representation theory of finite groups, general totally symmetric sets are not ``semisimple'' (i.e. direct sums of irreducibles), and the analysis of a reducible totally symmetric set involves a consideration of an ``extension problem'', essentially an analysis of nontrivial Jordan blocks.

\begin{definition}[(Ir)reducibility]
\label{definition:irreducible}
Let $\mathcal A \subset \End(V)$ be a totally symmetric set. $\mathcal A$ is said to be {\em reducible} if there is a proper subspace $W \le V$ (i.e. both dimension and codimension are positive) which is invariant under $\mathcal A$ and under the set of transformations $\{\rho(\sigma) \mid \sigma \in \Sigma_k\}$ for $\rho$ some realization map for $\mathcal A$. Such a subspace is said to be {\em $(\mathcal A, \rho)$ invariant}, for short. If no such $W$ exists, then $\mathcal A$ is said to be {\em irreducible}.
\end{definition}

\begin{construction}[Restrictions and quotients]\label{construction:restriction}
Let $\mathcal A \subset \End(V)$ be a totally symmetric set. Suppose that $W \le V$ is an $(\mathcal A, \rho)$-invariant subspace. Then the restrictions $A_i |_W$ form a totally symmetric set in $\End(W)$ with realization map given by the restrictions of $\rho(\sigma)$ to $W$. We call this the {\em restriction of $\mathcal A$ to $W$}, written $\mathcal A |_W$.

Likewise, the set of $A_i \in \mathcal A$ descend to maps $A_i/W \in \End(V/W)$ forming a totally symmetric set which we call the {\em quotient}, written $\mathcal A / W$.  With respect to a well-chosen basis, we can write 
\[
A_i = \begin{pmatrix} A_i|_W & X_i \\ 0 & A_i/W \end{pmatrix}.
\]
\end{construction}

\begin{remark}
If a totally symmetric set $\mathcal{A} \subset \End (V)$ leaves $W$ invariant but does not admit a realization map $\rho$ that leaves $W$ invariant, the restriction and quotient of $\mathcal{A}$ need not be totally symmetric (e.g. \Cref{example:simplexconst}).
\end{remark}

\begin{remark}
The restriction and/or quotient of a nondegenerate totally symmetric set may be degenerate (e.g. \Cref{example:simplexconst}).
\end{remark}

As a first consequence of the notion of irreducibility, we see that nondiagonalizability is a quintessentially {\em reducible} phenomenon.
\begin{proposition}\label{prop:diagonal}
Let $\mathcal A = \{A_1, \dots, A_k\} \subset \End(V)$ be an irreducible commutative totally symmetric set. Then $\mathcal A$ is diagonalizable---i.e. the set $\{A_1, \dots, A_k\}$ is simultaneously diagonalizable.
\end{proposition}
\begin{proof}
Since commuting linear endomorphisms preserve each other's eigenspaces, and since every linear endomorphism over $\C$ admits an eigenvector, it follows that the set of simultaneous eigenvectors for $\{A_1, \dots, A_k\}$ is nonempty. Let $E \le V$ be the span of such vectors. Note that the restriction to $E$ of any $A_i \in \mathcal A$ is diagonalizable. Certainly $E$ is $\mathcal A$-invariant. \Cref{lemma:conjformula} shows that $E$ is invariant under the action of all $\rho(\sigma)$ for any realization map $\rho$, and thus $\mathcal A$ admits a restriction to $E$ in the sense of \Cref{construction:restriction}. Since $\mathcal A$ is assumed to be irreducible, it follows that $E = V$, and hence each $\mathcal A$ is diagonalizable as claimed.
\end{proof}

\section{The induction construction}\label{section:induction}

Here we formulate a construction that lies at the heart of the classification scheme for irreducible totally symmetric sets, that of {\em induction}. Like its namesake in representation theory, the induction construction will allow us to expand the size of the totally symmetric set at the cost of enlarging the dimension of the vector space on which it acts. Following the description of induction given in \Cref{induction}, we illustrate it by presenting two examples of totally symmetric sets: the {\em standard construction} of \Cref{example:standard} and the {\em permutation construction} of \Cref{example:perm}. Ultimately we will see that the standard construction is the unique irreducible commutative totally symmetric set of $n$ elements in $\End(\C^n)$ (\Cref{maintheorem:max}).

\begin{convention}\label{convention:commute}
For the duration of \Cref{section:induction} and \Cref{section:proofA}, all totally symmetric sets $\mathcal A \subset \End(V)$ are assumed to be {\em commutative}.
\end{convention}

\begin{construction}[Induction]
\label{induction}
Let $\mathcal A \subset \End(V)$ be a totally symmetric set of cardinality $k$. Let $\rho: \Sigma_k \to \GL(V)$ be a realization map for $\mathcal A$, i.e. a set map such that 
\begin{equation}\label{equation:rho}
A_{\tau(i)}(v) = \rho(\tau) A_i\rho(\tau)^{-1}(v)
\end{equation}
holds for $i \in [k], v \in V$, and $\tau \in \Sigma_k$. Let $p \ge 1$ be an integer, and let $\lambda \in \C$ be arbitrary. Let $\subsetsof{k+p}{p}$ denote the set of $p$-element subsets of $[k+p]$, and let $\C[\subsetsof{k+p}{p}]$ denote the vector space with basis in bijection with $\subsetsof{k+p}{p}$.

The {\em induction of $\mathcal A$ from $\Sigma_k$ to $\Sigma_{k+p}$} is the totally symmetric set
\[
\Ind_{k}^{k+p}(\lambda)(\mathcal A) = \{\tilde A_1, \dots, \tilde A_{k+p}\}
\]
acting on the vector space $\C[\subsetsof{k+p}{p}] \otimes V$ via the following construction:

Define the set $[>k] = \{k+1, \dots, k+p\}$, and then choose a set $\{\sigma_S\}$ of elements of $\Sigma_{k+p}$ in bijection with $\subsetsof{k+p}{p}$ with the property that
\begin{equation}\label{equation:sigmaS}
\sigma_S(S) = [>k];
\end{equation}
take in particular $\sigma_{[>k]}$ to be the identity. For $i \in [>k]$, define $A_i \in \End(V)$ by $A_i = \lambda I$. Then define the elements $\tilde A_i$ for $i \in [k+p]$ of $\Ind_{k}^{{k+p}}(\lambda)(\mathcal A)$ by the formula
\begin{equation}\label{equation:tildeA}
\tilde A_i(S\otimes v) = S \otimes A_{\sigma_S(i)}(v).
\end{equation}

To see that $\Ind_{k}^{k+p}(\lambda)(\mathcal A)$ forms a totally symmetric set, we must check that the elements $\tilde A_i$ pairwise commute, and we must define a realization map $\tilde \rho: \Sigma_{k+p} \to \GL(\C[\subsetsof{k+p}{p}] \otimes V)$. The commutativity is easy to establish from the commutativity of $\mathcal A$ and \eqref{equation:tildeA}:
\[
\tilde A_i \tilde A_j(S\otimes v) = \tilde A_i(S \otimes A_{\sigma_S(j)}(v) = S \otimes A_{\sigma_S(i)} A_{\sigma_S(j)}(v) = S \otimes A_{\sigma_S(j)} A_{\sigma_S(i)}(v) = \tilde A_j \tilde A_i (S \otimes v).
\]

We next describe the realization map $\tilde \rho: \Sigma_{k+p} \to \GL(\C[\subsetsof{k+p}{p}] \otimes V)$. Define, for $\tau \in \Sigma_{k+p}$ and $S\otimes v\in \C[\subsetsof{k+p}{p}] \otimes V$,
\begin{equation}\label{equation:tilderho}
\tilde \rho(\tau)(S \otimes v) = \tau(S)\otimes \rho(\sigma_{\tau(S)} \tau \sigma^{-1}_{S})(v).
\end{equation}
Note that 
\[
\sigma_{\tau(S)} \tau \sigma^{-1}_{S}([>k]) = \sigma_{\tau(S)} \tau (S) = [>k],
\]
so that $\sigma_{\tau(S)} \tau \sigma^{-1}_{S} \in \Sigma_k \le \Sigma_{k+p}$, and so the application of $\rho$ to $\sigma_{\tau(S)} \tau \sigma^{-1}_{S}$ is sensible. The lemma below completes the induction construction by showing that $\tilde \rho$ as defined in \eqref{equation:tilderho} makes $\Ind_{{k}}^{{k+p}}(\lambda)(\mathcal A)$ into a totally symmetric set.

\begin{lemma}
The set map $\tilde \rho$ of \eqref{equation:tilderho} satisfies the formula
\[
\tilde A_{\tau(i)} (S\otimes v) = \tilde \rho(\tau) \tilde A_i \tilde \rho(\tau)^{-1}(S\otimes v)
\]
for all $\tau \in \Sigma_{k+p}, i \in [k+p], S \in \subsetsof{k+p}{p}$, and $v \in V$.
\end{lemma}
\begin{proof}
We will show that
\[
\tilde \rho(\tau) \tilde A_i (S\otimes v) = \tilde A_{\tau(i)} \tilde \rho(\tau) (S \otimes v).
\]
On the one hand, by \eqref{equation:tildeA}, \eqref{equation:tilderho} and then \eqref{equation:rho},
\begin{align*}
\tilde \rho(\tau) \tilde A_i (S\otimes v) &= \tilde \rho (\tau) \left(S \otimes A_{\sigma_S(i)}(v)\right)\\
&= \tau(S) \otimes \rho(\sigma_{\tau(S)} \tau \sigma^{-1}_{S})\left( A_{\sigma_S(i)}(v)\right)\\
&= \tau(S) \otimes A_{\sigma_{\tau(S)}\tau(i)}\left(\rho(\sigma_{\tau(S)} \tau \sigma^{-1}_{S})(v) \right).
\end{align*}
On the other hand, by \eqref{equation:tilderho} and then \eqref{equation:tildeA},
\begin{align*}
\tilde A_{\tau(i)} \tilde \rho(\tau) (S \otimes v) &= \tilde A_{\tau(i)} \left(\tau(S) \otimes \rho(\sigma_{\tau(S)} \tau \sigma^{-1}_{S}) \left(v\right) \right)\\
&= \tau(S) \otimes A_{\sigma_{\tau(S)}\tau(i)}\left( \rho(\sigma_{\tau(S)} \tau \sigma^{-1}_{S}) \left(v\right)\right)
\end{align*}
\end{proof}
\end{construction}

\begin{remark}
The construction of $\Ind_k^{k+p}(\lambda)(\mathcal A)$ depends on a non-canonical choice of the elements $\{\sigma_S \mid S \in \subsetsof{k+p}{p}\}$. It is straightforward to verify that different choices of such sets lead to isomorphic totally symmetric sets. 
\end{remark}

\begin{remark}\label{remark:diagonal}
If $\mathcal A$ acts diagonalizably on $V$ (e.g. if $\mathcal A$ is irreducible), then the construction shows that $\Ind_k^{k+p}(\lambda)(\mathcal A)$ acts diagonalizably on $\C[\subsetsof{k+p}{p}]\otimes V$.
\end{remark}

\para{Examples} In the following examples, we will adopt a compact notation scheme for depicting a diagonalizable totally symmetric set as a tableau. Each row corresponds to an element of the totally symmetric set, and each column corresponds to an element of a simultaneous eigenbasis; the entries index the corresponding eigenvalues. Where helpful, we will add an additional row along the top indicating the subsets indexing basis elements. 

The realization maps in this setting can be chosen to arise via the action of permutation matrices, implying that in all tableau below, every permutation of the rows can be realized by a permutation of the columns. A similar construction appears in \cite{Kordek-Margalit}, which classifies square matrices satisfying the property that every permutation of the rows is realized by a permutation of the columns. The proof of \cite[Lemma 2.5]{Kordek-Margalit}, properly interpreted in our context, can be understood to say that every irreducible totally symmetric set with cardinality equal to the dimension is isomorphic to the {\em standard construction}, as defined below. 




\begin{example}[The standard totally symmetric set]\label{example:standard}
Let $\mathcal A = \{A_1, A_2\} \subset \End(\C)$ be the degenerate totally symmetric set depicted in tableau form as shown:
\[
\mathcal A = \begin{array}{c} 1\\ 1 \end{array}.
\]
The induction $\tilde A = \Ind_2^3(2)(\mathcal A)$ then has tableau
\[
\tilde A = \begin{array}{ccc}
            \{1\}   & \{2\} & \{3\} \\ \hline
            2       &   1   &   1     \\
            1       &   2   &   1    \\
            1       &   1   &   2    \\

\end{array}.
\]
We could also obtain $\tilde A$ by starting with the one-element totally symmetric set $\{(2)\}$ and applying $\Ind_1^3(1)$.

More generally, letting $\mathcal A^{triv}_{k}(\lambda)$ be the $k$-element degenerate irreducible totally symmetric set with eigenvalue $\lambda$, we call a totally symmetric set of the form $\Ind_k^{k+1}(\nu)(\mathcal A^{triv}_{k}(\lambda))$ (with $\nu \ne \lambda$) {\em standard} and write $\mathcal A_k^{std}(\lambda, \nu)$.
\end{example}
\begin{example}[Permutation-type totally symmetric sets]\label{example:perm}
Let $\mathcal A = \{A_1\} \subset \End(\C)$ be the degenerate totally symmetric set given by $A_1 = 1$. We first construct $\mathcal A_2:= \Ind_1^2(2)(\mathcal A)$, an instance of the standard construction:
\[
\mathcal A_2 = \begin{array}{cc}
            e_1     & e_2   \\ \hline 
            2       & 1     \\
            1       & 2
\end{array}.
\]
Next we construct $\mathcal A_3 = \Ind_2^3(3)(\mathcal A_2)$:
\[
\mathcal A_3 = \begin{array}{cccccc}
    \{1\}\otimes e_1& \{1\}\otimes e_2  &     \{2\}\otimes e_1& \{2\}\otimes e_2  &    \{3\}\otimes e_1& \{3\}\otimes e_2\\ \hline
    3               &   3               &   1   &   2   & 2     & 1\\
    2               &   1               &   3   &   3   & 1     & 2\\
    1               &   2               &   2   &   1   & 3     & 3
\end{array}.
\]
One can of course repeat this construction as many times as desired. A more compact description is as follows: one chooses an injective function $\lambda: [k] \to \C$, and then defines $\mathcal A(\lambda)$ as the totally symmetric set acting on $\C[\Sigma_k]$ via the formula
\[
A_i(\sigma) = \lambda(\sigma(i))\sigma,
\]
with $\rho(\tau)$ acting on $\C[\Sigma_k]$ by right-multiplication by $\tau^{-1}$. We call such totally symmetric sets {\em of permutation type}.
\end{example}

\begin{example}[The partition construction]\label{example:partition}
The construction of permutation-type totally symmetric sets in the previous example admits a broad generalization. Let $\kappa$ be a partition of $[k]$, which we view as an equivalence relation $\sim_\kappa$ on $[k]$. Let the equivalence classes have size $k_1, \dots, k_m$. Let $\vec\lambda: [k] \to \C$ be a function with the property that
\[
\vec\lambda(a) = \vec\lambda(b) \iff a \sim_\kappa b.
\]
We call such a function $\vec \lambda$ a {\em weight}.

We define a totally symmetric set called the {\em partition construction} $\mathcal A(\vec\lambda)$ associated to $\vec\lambda$. The vector space $V_{\vec\lambda}$ is defined to be the space spanned by functions of the form $\vec \lambda \circ \sigma$ for $\sigma \in \Sigma_k$. We identify two such functions $\vec\lambda \circ \sigma$ and $\vec \lambda \circ \tau$ if they agree everywhere, so that 
\[
\dim(V_{\vec \lambda}) = \binom{k}{k_1, \dots, k_m} = \frac{k!}{k_1!\dots k_m!},
\]
where $\binom{k}{k_1,\dots,k_m}$ denotes the {\em multinomial coefficient}. The element $A_i \in \End(V_{\vec \lambda})$ of $\mathcal A(\vec\lambda)$ acts diagonally on $V_{\vec \lambda}$ by the formula
\[
A_i(\vec \lambda \circ \sigma) = \vec \lambda(\sigma(i)) \lambda \circ \sigma.
\]
Note that this is totally symmetric, with $\tau \in \Sigma_k$ acting on $\lambda \circ \sigma$ via precomposition with $\tau^{-1}$: 
\[
(\tau \cdot (\lambda \circ \sigma))(i) = (\lambda \circ \sigma)(\tau^{-1}(i)).
\]
\end{example}
The following lemma is evident from the induction and partition constructions, but is crucial in what follows.
\begin{lemma}[Partitions as repeated induction]
\label{lemma:indpart}
Let $\lambda_1, \dots, \lambda_m \in \C$ be distinct, and let $k_1, \dots, k_m$ be positive integers; for $1 \le j \le m$ define $K_j:= k_1 + \dots + k_j$ (taking $K_0 =0$). Construct a sequence of totally symmetric sets $\mathcal A_j$ for $0 \le j \le m$ (with $\mathcal A_0$ the zero-element totally symmetric set acting on $\C^0$) recursively via
\[
\mathcal A_{j+1} = \Ind_{K_j}^{K_{j+1}}(\mathcal A_j).
\]
Then $\mathcal A_{m}$ is isomorphic to the partition construction $\mathcal A(\vec \lambda)$ for the weight function $\vec \lambda: [k] \to \C$ given by $\vec \lambda(i) = \lambda_j$ for $K_{j-1} < i \le K_j$.

Conversely, all partition constructions arise via repeated induction.
\end{lemma}

\section{Classification of irreducible totally symmetric sets}\label{section:proofA}
We come now to the first main result: a complete classification of irreducible totally symmetric sets. \Cref{mainthm:summary} is proved here as \Cref{corollary:partition} of \Cref{mainthm:irred}. The analysis is based on the notion of {\em depth} of an eigenvalue, discussed in \Cref{subsection:depth}. Depth allows one to establish an important converse to \Cref{mainthm:irred}, namely that the induction of an irreducible totally symmetric set is (in general) irreducible; this appears in \Cref{subsection:depth} as \Cref{prop:indirred}. Having established the notion of depth, the proofs of \Cref{mainthm:irred} and \Cref{corollary:partition} follow in \Cref{subsection:proofA}.

Recall that \Cref{convention:commute} is still in effect: in this section, all totally symmetric sets are assumed to be commutative.

\subsection{Depth}\label{subsection:depth}
Let $\mathcal A \subset \End(V)$ be a totally symmetric set of cardinality $k$, and let $\lambda$ be an eigenvalue of $\mathcal A$. Briefly, the {\em depth} of $\lambda$ records the size of the largest subset $S \subset [k]$ for which the $p$-fold eigenspace $E_\lambda^S$ has positive dimension. \Cref{lemma:depth} and \Cref{definition:depth} make this precise.
\begin{lemma}
\label{lemma:depth}
Let $\mathcal A \subset \End(V)$ be a totally symmetric set of cardinality $k$. Fix $j \le k$, and let $\lambda$ be an eigenvalue of $\mathcal A$. Let $S \subset [k]$ be any subset of cardinality $j$. Then the dimension of $E_\lambda^S$ is an integer $\mu_\lambda (j)$ that depends only on $j$.
\end{lemma}
\begin{proof}
By \Cref{lemma:conjformula}, for subsets $S, S'$ of equal cardinality, the spaces $E_\lambda^S$ and $E_\lambda^{S'}$ are isomorphic.
\end{proof}

We extend the definition of the integers $\mu_\lambda(j)$ to $j > k$ by taking $\mu_\lambda(j) = 0$ for $j > k$.

\begin{definition}[Depth]
\label{definition:depth}
Let $\mathcal A = \{A_1, \dots, A_k\} \subset \End(V)$ be a totally symmetric set. The {\em depth} of an eigenvalue $\lambda$ of $\mathcal A$ is the least positive integer $p_\lambda$ such that $\mu_\lambda(p_\lambda) > 0$, but $\mu_\lambda(p_\lambda+1) = 0$.
\end{definition}

\begin{lemma}
Suppose $\mathcal A$ is an irreducible nondegenerate totally symmetric set of cardinality $k$. Then every eigenvalue of $\mathcal A$ has depth $p < k$.
\end{lemma}
\begin{proof}
Let $\mathcal A \subset \End(V)$ be a totally symmetric set of cardinality $k$, and suppose that $\lambda$ is an eigenvalue of depth at least $k$. Then $E_\lambda^{[k]} \subset V$ is a proper $(\mathcal A, \rho)$-invariant subspace on which $\mathcal A$ restricts as a degenerate totally symmetric set, contrary to the assumption that $\mathcal A$ is irreducible and nondegenerate.
\end{proof}

\para{Induction and irreducibility} Depth allows us to address a lingering question in the theory of inductions---when is the induction of an irreducible totally symmetric set itself irreducible? \Cref{prop:indirred} shows that this is the case as long as one uses a ``new'' eigenvalue; this is proved by considering the maximal-depth eigenspaces. The crucial observation from the induction construction is recalled below as \Cref{remark:depth}. 

\begin{remark}
\label{remark:depth}
One sees from the construction of $\Ind_k^{k+p}(\lambda)(\mathcal A)$ (\Cref{induction}) that $\lambda$ has depth at least $p$, and that the depth equals $p$ if and only if $\lambda$ is not an eigenvalue of $\mathcal A$. Moreover, $\C[\subsetsof{k+p}{p}] \otimes V$ is spanned by the collection of $p$-fold eigenspaces $E_\lambda^S$ as $S$ ranges over $\subsetsof{k+p}{p}$, and this sum is direct if and only if $\lambda$ is not an eigenvalue of $\mathcal A$.
\end{remark}

\begin{proposition}\label{prop:indirred}
Let $\mathcal A \subset \End(V)$ be a totally symmetric set of cardinality $k$. Then for any $p > 0$ and any $\lambda$ distinct from all eigenvalues of $\mathcal A$, the induction $\Ind_k^{k+p}(\lambda)(\mathcal A)$ is irreducible if and only if $\mathcal A$ is irreducible. 
\end{proposition}
\begin{proof}
First, suppose $\mathcal A$ is reducible; let $W \le V$ be a proper $(\mathcal A, \rho)$-invariant subspace. Then it is clear from the induction construction (\Cref{induction}) that $\C[\subsetsof{k+p}{k}]\otimes W$ is a proper $(\Ind_k^{k+p}(\lambda)(\mathcal A), \tilde \rho)$-invariant subspace which witnesses the reducibility of $\Ind_k^{k+p}(\lambda)(\mathcal A)$.

Conversely, suppose that $\Ind_k^{k+p}(\lambda)(\mathcal A)$ is reducible, and let $W \le \C[\subsetsof{k+p}{p}]\otimes V$ be a proper $(\tilde A, \rho)$-invariant subspace for some realization map $\rho$ (not necessarily of the form arising in \Cref{induction}). We will show that $\mathcal A$ is reducible.

Since $\lambda$ is not an eigenvalue of $\mathcal A$, \Cref{remark:depth} provides for a decomposition
\begin{equation}
    \label{directsum}
\C[\subsetsof{k+p}{p}] \otimes V = \bigoplus_{S \in \subsetsof{k+p}{p}} E_\lambda^S.
\end{equation}
By \Cref{lemma:conjformula}, the subspace $E_\lambda^{[>k]}$ is invariant under the action of the realization map $\rho$ when restricted to $\sigma \in \Sigma_k \le \Sigma_{k+p}$; it is also invariant under the action of $\{\tilde A_1, \dots, \tilde A_k\}$. It is straightforward to see that the restriction of $\{\tilde A_1, \dots, \tilde A_k\}$ to $E_\lambda^{[>k]}$ is isomorphic to $\mathcal A$ in the sense of \Cref{definition:iso}.

Let $P: \C[\subsetsof{k+p}{p}] \otimes V \to E_\lambda^{[>k]}$ be the projection onto $E_\lambda^{[>k]}$ afforded by \eqref{directsum}. We claim that $P(W)$ is a proper subspace of $E_\lambda^{[>k]}$ invariant under the restriction of $\rho$ to $\Sigma_k$ and the restrictions of $\{\tilde A_1, \dots, \tilde A_k\}$ to $E_\lambda^{[k]}$,  which will establish the reducibility of $\mathcal A$.

This will follow from a closer analysis of the projection map $P$. Since $\tilde A_i$ is diagonalizable, the map
\[
P_j = \prod_{\nu \ne \lambda}\frac{1}{\lambda -\nu}\left(\tilde A_j - \nu I\right)
\]
(where the product ranges over all eigenvalues of $\tilde A_j$ other than $\lambda$) is the projection of $\C[\subsetsof{k+p}{p}]\otimes V$ onto $E_\lambda^j$. Thus, $P$ admits an expression of the form
\[
P = \prod_{j > k} P_j.
\]
Since each $\tilde A_i$ and $\nu I$ preserves $W$, it follows that $P$ preserves $W$. Thus, $P(W)$ can be written as the intersection
\begin{equation}
    \label{equation:pw}
P(W) = W \cap E_\lambda^{[>k]}.
\end{equation}
As each of $W$ and $E_\lambda^{[k]}$ are invariant under the restriction of $\rho$ to $\Sigma_k$ and $\{\tilde A_1, \dots, \tilde A_k\}$ to $E_\lambda^{[k]}$, it follows that $P(W)$ is invariant under these transformations as well. 

It remains to show that $P(W)$ is a proper subspace of $E_\lambda^{[>k]}$. If $P(W) = E_\lambda^{[>k]}$, then \eqref{equation:pw} shows that $E_\lambda^{[>k]} \le W$. Then by \Cref{lemma:conjformula}, $W$ must contain all subspaces $E_\lambda^S$, in which case $W$ would not be a proper subspace of $\C[\subsetsof{k+p}{p}] \otimes V$. To see that $P(W)$ contains some nonzero $v \in E_\lambda^{[>k]}$, let $w \in W$ be nonzero. In the coordinates on $\C[\subsetsof{k+p}{p}] \otimes V$ of \eqref{directsum}, 
\[
w = \sum S \otimes v_S
\]
for elements $v_S \in E_\lambda^S$, not all zero. Let $S \subset [k+p]$ be such that $v_S \ne 0$, and let $\sigma_S \in \Sigma_{k+p}$ have the property $\sigma_S(S) = [>k]$. Then by \Cref{lemma:conjformula}, the $E_\lambda^{[>k]}$-coordinate of $\rho(\sigma_S)(w) \in W$ is nonzero, and the projection of $\rho(\sigma_S)(w)$ to $E_\lambda^{[>k]}$ is therefore nonzero as well.
\end{proof}

\subsection{Classification of irreducible totally symmetric sets}\label{subsection:proofA}
\Cref{mainthm:irred} establishes that an irreducible totally symmetric set arises from the induction construction. By an inductive argument, we will use this to show in \Cref{corollary:partition} that there are finitely many classes of irreducible totally symmetric sets of cardinality $k$, indexed by the partitions of $k$. One recovers the partition associated to an irreducible totally symmetric set by fixing any simultaneous eigenvector $v$ for $\mathcal A$ and looking at the multiplicities of the associated $k$ eigenvalues. 

To see how to express an arbitrary irreducible totally symmetric set as an induction, we begin by identifying the totally symmetric set from which it will be induced.

\begin{lemma}\label{lemma:restriction}
Let $\mathcal A = \{A_1, \dots, A_k\} \subset \End(V)$ be a totally symmetric set with eigenvalue $\mu$, and let $S \subset [k]$ be a subset. Then the set of elements
\[
\{A_i: i \not \in S\}
\]
restricts to a totally symmetric set on $E_\mu^S$ which is denoted
\[
\mathcal A'(S).
\]
In the special case where $\mu$ has depth $p$ and $S = [p]$, we write $\mathcal A' \setminus \mu$. In this case, the eigenvalues of $\mathcal A' \setminus \mu$ are exactly those of $\mathcal A$, excluding $\mu$.
\end{lemma}
\begin{proof}
Since the elements of $\mathcal A$ pairwise commute, the space $E_\mu^S$ is $A_i$-invariant for all $i \in [k]$; in particular this holds for $i \in S^c$. It remains to exhibit a suitable realization map, i.e. one for which $E_\mu^S$ is invariant under the action of $\rho(\sigma)$ for every $\sigma$ that fixes $S$. In fact we will see that {\em any} realization map for $\mathcal A$ will suffice.

Let $\sigma$ be a permutation of $S^c \subset [k]$; we extend $\sigma$ to a permutation $\tilde \sigma$ of $[k]$ that fixes $S$ pointwise. Let $\rho: \Sigma_k \to \GL(V)$ be a realization map for $\mathcal A$. By \Cref{lemma:conjformula}, for such $\tilde \sigma$, the automorphism $\rho(\tilde \sigma)$ restricts to an automorphism of $E_\mu^S$, and induces the permutation $\sigma$ on the set $\{A_i|_{E_\mu^S}: i \not \in S\}$ as required.
\end{proof}

\begin{theorem}\label{mainthm:irred}
Let $\mathcal A \subset \End(V)$ be an irreducible totally symmetric set of cardinality $k$, and let $\mu$ be an eigenvalue of $\mathcal A$ of depth $p$. Then $\mathcal A$ is isomorphic as a totally symmetric set to an induction:
\[
\mathcal A \cong \Ind_{{k-p}}^{k}(\mu)\left( \mathcal A' \setminus \mu\right).
\]
\end{theorem}
\begin{remark}
An irreducible totally symmetric set with exactly one eigenvalue is necessarily the degenerate totally symmetric set acting on $\C$. This can be viewed as an induction from the empty totally symmetric set; in particular, no special modifications need to be made to the statement of \Cref{mainthm:irred} for it to hold in this case.
\end{remark}
\begin{proof}
Following \Cref{definition:iso} and \Cref{lemma:restriction}, we must construct an isomorphism 
\[
T: V \to \C[\subsetsof{k}{p}]\otimes E_\mu^{[p]}
\]
that conjugates $\mathcal A$ to $\Ind_{{k-p}}^{k}(\mu)(\mathcal A' \setminus \mu)$.

We first claim that there is a decomposition
\begin{equation}\label{equation:vdecomp}
V \cong \bigoplus_{S \in \subsetsof{k}{p}} E_\mu^S.
\end{equation}
To see this, we observe that by the definition of depth (\Cref{definition:depth}), the $p$-fold eigenspaces $E_\mu^S$ and $E_\mu^{S'}$ intersect trivially for distinct $p$-element subsets $S, S' \subset [k]$. Let $E_\mu$ denote the direct sum of $E_\mu^S$ for $S$ a $p$-element subset, viewed as a subspace of $V$; it remains to show that $E_\mu = V$.  \Cref{lemma:conjformula} shows that $E_\mu$ is invariant under the action of the realization map $\rho$, and so the restriction of $\mathcal A$ to $E_\mu$ forms a totally symmetric set. As $\mathcal A$ is assumed to be irreducible, it follows that $E_\mu = V$. 

Having established \eqref{equation:vdecomp}, it remains to identify the direct sum with the tensor product $\C[\subsetsof{k}{p}]\otimes E_\mu^{[p]}$. To do so, we must choose a set of isomorphisms identifying the various $E_\mu^S$ with the fixed space $E_\mu^{[p]}$. Choose a set 
\[
\left \{\sigma_S \mid S \in \subsetsof{k}{p}\right \} \subset \Sigma_k
\]
with the property that $\sigma_S(S) = [p]$. Then by \Cref{lemma:conjformula}, the automorphism $\rho(\sigma_S) \in \GL(V)$ restricts to an isomorphism 
\[
\rho(\sigma_S): E_\mu^S \to E_\mu^{[p]}.
\]
This leads to the definition of the isomorphism $T: V \to \C[\subsetsof{k}{p}] \otimes E_\mu^{[p]}$. Via the decomposition \eqref{equation:vdecomp}, it suffices to specify $T$ for vectors $v_S \in E_\mu^S$ for $S \in \subsetsof{k}{p}$; for $v_S \in E_\mu^S$, take
\begin{equation}
    \label{equation:T}
    T(v_S) = S \otimes \rho(\sigma_{S})(v_S).
\end{equation}

It remains to see that $T$ conjugates $\mathcal A$ to $\Ind_{k-p}^k(\mu)(\mathcal A' \setminus \mu)$, or equivalently that $T^{-1}$ does so in reverse. Recall from \Cref{lemma:restriction} that $\mathcal A' \setminus \mu$ is the totally symmetric set given by restricting $A_i$ for $i>p$ to the $p$-fold eigenspace $E_\mu^{[p]}$; denote the restriction of $A_i$ by $B_i$. Then let
\[
\{\tilde B_1, \dots, \tilde B_k\} \subset \End\left(\C[\subsetsof{k}{p}] \otimes E_\mu^{[p]}\right)
\]
be the output of the induction construction $\Ind_{k-p}^k(\mu)$ as applied to $\mathcal A' \setminus \mu$, using the same set $\{\sigma_S\}$ of coset representatives as used to construct $T$. Then for $i \in [k]$ and an element $v_S \in E_\mu^S \leq V$,
\begin{align*}
T^{-1} \tilde B_i T (v_S) &= T^{-1} \tilde B_i\left(S \otimes \rho(\sigma_{S})(v_S)\right)\\
            &= T^{-1} \left(S \otimes B_{\sigma_{S}(i)}\left(\rho(\sigma_{S})(v_S)\right)\right)\\
            &= T^{-1} \left(S\otimes A_{\sigma_{S}(i)}\left(\rho(\sigma_{S})(v_S)\right)\right)\\
            &= T^{-1} \left(S \otimes \rho(\sigma_{S}) A_{i}(v_S)\right)\\
            &= A_i(v_S),
\end{align*}
the last equality holding by \eqref{equation:T} as applied to $A_i(v_S)$ (recall that $v_S$ is contained in the $A_i$-invariant subspace $E_\mu^S$).
\end{proof}



\begin{corollary}[\Cref{mainthm:summary}]
\label{corollary:partition}
Let $\mathcal A \subset \End(\C^n)$ be an irreducible commutative totally symmetric set of cardinality $k$. Then there is a weight $\vec \lambda$ for which 
\[
\mathcal A \cong \mathcal A (\vec \lambda).
\]
The dimension of the space $V_{\vec \lambda} \cong \C^n$ on which $\mathcal A(\vec \lambda)$ acts is given by the multinomial coefficient $\binom{k}{k_1,\dots, k_p}$. In particular, if $\mathcal A$ is nondegenerate then $k \le n \le k!$, with the unique totally symmetric set of dimension $k$ given by the standard construction (\Cref{example:standard}), and the unique totally symmetric set of dimension $k!$ given by the permutation construction (\Cref{example:perm}).
\end{corollary}
\begin{proof}
We proceed by induction on $k$, the base case $k = 0$ being trivial. By \Cref{mainthm:irred}, $\mathcal A \cong \Ind_{k-p}^k(\mu)(\mathcal A' \setminus \mu)$, and by \Cref{prop:indirred}, irreducibility of $\Ind_{k-p}^k(\mu)(\mathcal A' \setminus \mu)$ implies the irreducibility of $\mathcal A' \setminus \mu$. By the inductive hypothesis, $\mathcal A' \setminus \mu$ arises from the partition construction for some weight function $\vec \lambda': [k-p] \to \C$. By \Cref{lemma:indpart}, 
\[
\mathcal A \cong \Ind_{k-p}^k(\mu)(\mathcal A' \setminus \mu) \cong \Ind_{k-p}^k(\mu)(\mathcal A(\vec \lambda ')) \cong \mathcal A(\vec\lambda)
\]
for an extension $\vec \lambda: [k] \to \C$ of the weight $\vec \lambda'$.
\end{proof}

\begin{example}
We use \Cref{corollary:partition} to classify all irreducible totally symmetric sets of at most $4$ elements. The case $k = 1$ is trivial. There are exactly two types for $k =2$: the degenerate $\mathcal A_2^{triv}(\lambda)$ in $\End(\C)$ with associated partition $2$ and the standard $\mathcal A_2^{std}(\lambda, \nu)$ in $\End(\C^2)$ with associated partition $1 \le 1$. 

For $k = 3$ we have three partitions:
\[
\begin{array}{c|cc}
\mbox{name} & \mbox{partition}        & \mbox{dimension}\\ \hline 
\mbox{triv}        & 3                  &   1\\
\mbox{std}         & 1 \le 2            &   3\\
\mbox{perm}        & 1\le 1 \le 1       &   6
\end{array}.
\]

For $k = 4$ there are five partitions:
\[
\begin{array}{c|cc}
\mbox{name}             & \mbox{partition}        & \mbox{dimension}\\ \hline 
\mbox{triv}             & 4                         &   1\\
\mbox{std}              & 1\le 3                    &   4\\
\Ind_2^4(\mbox{triv})   & 2 \le 2                   &   6\\
\Ind_2^4(\mbox{std})    & 1 \le 1 \le 2             &   12\\
\mbox{perm}             & 1\le 1 \le 1 \le 1        &  24
\end{array}.
\]
\end{example}

\section{Totally symmetric arrangements}\label{section:arrangement}
We turn now to the second half of the paper, with the overall objective of proving \Cref{theorem:sizebound} and \Cref{maintheorem:max}. Accordingly, we relax the assumption of the previous sections that $\mathcal A$ be a {\em commutative} totally symmetric set: {\em \Cref{convention:commute} is no longer in effect.} 

Our study of non-commuting totally symmetric sets will be mediated through a study of the associated arrangements of (generalized) eigenspaces. Here in \Cref{section:arrangement}, we establish some of the basic theory of totally symmetric arrangements. The content here is largely expository: we introduce several key technical notions: the {\em stabilizer subgroup} of an arrangement (\Cref{defintion:arrstab}), the {\em system of subrepresentations construction} (\Cref{subrepsystem}), and the {\em eigenspace construction} (\Cref{espace}). We intertwine this with a discussion of our key examples of totally symmetric arrangements: the {\em simplex arrangement} (\Cref{example:simplex}) and the {\em $\tilde \Sigma_5$ arrangement} (\Cref{example:s5tilde}).

\subsection{Basic notions for totally symmetric arrangements}
\begin{definition}[Totally symmetric arrangement, isomorphism]
\label{definition:tsa}
Let $V$ be a vector space. Succinctly, a {\em totally symmetric arrangement} is a totally symmetric set in the $\GL(V)$-set $\Gr_d(V)$. 

Less tersely, let $\mathcal W = \{W_1, \dots, W_k\}$ be a set of $d$-dimensional vector subspaces of $V$ (henceforth simply called $d$-planes), not necessarily distinct. $\mathcal W$ is said to be a {\em totally symmetric arrangement} if for any permutation $\sigma \in \Sigma_k$, there is an element $P_\sigma \in \GL(V)$ such that the equation
\[
P_{\sigma} W_i = W_{\sigma(i)}
\]
holds for all $i \in [k]$. An arrangement with each $W_i = W$ for some fixed $d$-plane $W$ is said to be {\em degenerate}. We remind the reader that ``collision implies collapse'' (\Cref{remark:cimpc}): if $W_i = W_j$ for some pair of indices $i \ne j$, then the arrangement $\{W_i\}$ is degenerate.

We say that arrangements $\mathcal W = \{W_i\} \subset V$ and $\mathcal W'=\{W_i'\} \subset V'$ are {\em isomorphic} if there is a linear isomorphism $T: V\to V'$ such that the restriction of $T$ to each $W_i \in \mathcal W$ induces an isomorphism $T: W_i \to W_i'$.
\end{definition}

\begin{definition}[Dual arrangement]
\label{definition:dual}
Let $\mathcal W  = \{W_i\}$ be a totally symmetric arrangement in $V$. The {\em dual arrangement} $\mathcal W^* \subset V^*$ is the arrangement of dual planes to the elements $W_i$, i.e. the subspaces
\[
W_i^* = \{\alpha \in V^* \mid \alpha|_{W_i} = 0 \}.
\]
\end{definition}

\begin{definition}[Reduced arrangement]
\label{definition:quotientarrangement}
Let $\mathcal W = \{W_i\}$ be a totally symmetric arrangement in $V$. Define
\[
Q:= \bigcap_{i \in [k]}W_i.
\]
$W$ is {\em reduced} if $Q = \{0\}$, and is {\em reducible} otherwise. The {\em reduced arrangement} $\mathcal W^{red}$ is the totally symmetric arrangement in $V/Q$ given by
\[
\mathcal W^{red} = \{\overline{W_i}\},
\]
where $\overline{W_i}$ denotes the projection of $W_i$ to $V/Q$.
\end{definition}

\begin{lemma}\label{lemma:nondegquot}
If the arrangement $\mathcal W$ is nondegenerate, so is the reduced arrangement $\mathcal W^{red}$.
\end{lemma}
\begin{proof}
As $\mathcal W$ is nondegenerate, for any pair of distinct indices $i\ne j$, there is an element $v \in W_i \setminus W_j$. As $W_j \le V$ is a vector subspace, the coset $v + Q\subset W_i$ is disjoint from $W_j$, hence witnesses the fact that $\overline{W_i}$ and $\overline{W_j}$ are distinct in $V/Q$.
\end{proof}

\para{The stabilizer subgroup} Any arrangement of subspaces $\mathcal W = \{W_1, \dots, W_k\} \subset V$ (even in the absence of total symmetry) carries an associated {\em stabilizer subgroup}. An understanding of the structure of this stabilizer for special arrangements will be useful in the classification arguments to come, c.f. \Cref{lemma:stabstd,lemma:stabs5}.

\begin{definition}[Stabilizer of an arrangement]\label{defintion:arrstab}
Let $\mathcal W = \{W_1, \dots, W_k\} \subset V$ be an arrangement of subspaces. The {\em stabilizer subgroup} is defined as
\[
\Stab(\mathcal W) = \{A \in \GL(V) \mid A W_i = W_i\mbox{ for all }W_i \in \mathcal W\}.
\]
\end{definition}

\para{Systems of subrepresentations} Here we present a general mechanism for constructing totally symmetric arrangements from a representation of a group $G$ given as an extension of $\Sigma_k$.

\begin{construction}[System of subrepresentations]\label{subrepsystem}
Let $G$ be a group admitting a surjective homomorphism $\psi: G \to \Sigma_k$, and let $V$ be a $G$-representation. Define the subgroup $G_1 \le G$ as the stabilizer of $1$ under the action of $G$ on $[k]$ induced by $\psi$.

One can construct a totally symmetric arrangement from this setup as follows. Choose $W_1 \le V$ to be a $G_1$-subrepresentation, and then define $W_i = gW_1$ for any $g \in G$ such that $\psi(g)\cdot 1 = i$. By definition of $G_1$, this is well-defined. This is totally symmetric, with realization map given by a set-theoretic section $\rho: \Sigma_k \to G$ of $\psi$.
\end{construction}

\subsection{Two special arrangements}
Here we apply the system of subrepresentations construction in two important examples; later we will see that these are the only totally symmetric arrangements of maximal possible size.

\para{The simplex arrangement} First, a comment on the terminology: while the mechanism underlying the simplex arrangement is the standard representation of the symmetric group, the term ``standard arrangement'' would clash with the ``standard construction'' of \Cref{example:standard}; we choose ``simplex arrangement'' instead to reflect the fact that the arrangement can be constructed from the vertices (or dually, the faces) of a regular $n$-simplex in $\R^n$ (albeit in different coordinates).

\begin{example}[The (dual) simplex arrangement]
\label{example:simplex}
Concisely, the simplex arrangement and its dual arise as the two possible systems of subrepresentations for $V = V_{std}^n$, the standard representation of $\Sigma_{n+1}$ (the indexing convention is admittedly unusual but is chosen so that $\dim(V_{std}^n) = n$ which will be more convenient for our purposes). There is a decomposition
\[
V_{std}^n = V_{std}^{n-1} \oplus \C
\]
of $V_{std}^n$ as a $\Sigma_{n}$-representation, where we embed $\Sigma_{n}$ into $\Sigma_{n+1}$ as the stabilizer of $1$, as required for the system of subrepresentations construction. Choosing $W_1 = \C$ in this decomposition gives an arrangement of $n+1$ lines in $V_{std}^n \cong \C^n$ which we call the {\em simplex arrangement} $\mathcal S_n$, and choosing $W_1 = V_{std}^{n-1}$ gives an arrangement of $n+1$ hyperplanes in $\C^n$ which we call the {\em dual simplex arrangement} $\mathcal S_n^*$ (indeed, $V_{std}^n$ is canonically self-dual, and these arrangements are identified with each other under the dualizing map).

For later use, it will be helpful to have an explicit coordinate representation. Equip $\C^{n+1}$ with basis $e_1, \dots, e_{n+1}$, and define the $\Sigma_{n+1}$-invariant vector
\[
\gamma = e_1+ \dots + e_{n+1}.
\]
$\C^{n+1}$ carries the permutation action of $\Sigma_{n+1}$, and we define in this way
\[
V_{std}^{n}:= \C^{n+1}/\pair{\gamma}.
\]

The simplex arrangement is given as the arrangement of $k = n+1$ lines in $V^n_{std}$ with $\ell_i$ spanned by the image of $e_i$ in $V_{std}^n$ for $1 \le i \le n+1$ (we will abuse notation and continue to refer to these vectors as $e_i$).

To represent the dual simplex arrangement, we see that the dual space $(V_{std}^n)^*$ is given as the subspace of $(\C^{n+1})^* = \pair{e_1^*, \dots, e_{n+1}^*}$ determined by the condition 
\[
(V_{std}^n)^* = \{\alpha \in (\C^{n+1})^* \mid \alpha(\gamma) = 0\}.
\]
Defining
\[
\gamma^* = e_1^* + \dots + e_{n+1}^*,
\]
one sees that the dual vectors
\[
\alpha_i = (n+1)e_i^* - \gamma^*
\]
are contained in $(V_{std}^n)^*$ and are totally symmetric under the dual action of the standard representation. Thus the arrangement $\ker(\alpha_i)$ is an explicit representation of the dual simplex arrangement in $V_{std}^n$.
\end{example}

\para{Structure of the simplex arrangement} We begin with a general observation. We recall that a $d$-plane $W \subset \C^n$ can be identified with a coset 
\[
W = M_W \GL_d(\C),
\]
where $M_W$ is a $n \times d$ matrix of rank $d$ whose columns form a basis for $W$. Given two such $d$-planes $W_1, W_2 \subset \C^n$, the intersection $W_1 \cap W_2$ is canonically identified with the kernel of the juxtaposition matrix $(M_{W_1}\mid M_{W_2})$ in the following way: the element $M_{W_1}v_1 = M_{W_2}v_2$ corresponds to $(v_1,-v_2) \in \ker (M_{W_1}\mid M_{W_2})$. In particular, $W_1 \cap W_2 = \{0\}$ if and only if $(M_{W_1}\mid M_{W_2})$ is injective.

\begin{proposition}
\label{proposition:simpmax}
Let $\mathcal L = \{\ell_1, \dots, \ell_k\}$ be a totally symmetric arrangement of lines in $\C^n$. Then $k \le n+1$, and if $k = n+1$ then $\mathcal L$ is isomorphic to the simplex arrangement $\mathcal S_n$. Likewise, any totally symmetric arrangement of hyperplanes in $\C^n$ has $k \le n+1$ elements, and if $k = n+1$, it is isomorphic to the dual simplex arrangement $\mathcal S_n^*$.
\end{proposition}
\begin{proof}
We assume that $k \ge n+1$; otherwise there is nothing to show. Without loss of generality, we can assume that $\mathcal L$ is not contained in any proper subspace of $\C^n$. Renumbering the elements of $\mathcal L$ if necessary, we can therefore take vectors $e_i' \in \ell_i$ for $1 \le i \le n$ such that $\{e_1',\dots, e_n'\}$ forms a basis for $\C^n$. Let 
\[
v = \sum_{i = 1}^n a_i e_i'
\]
be some nonzero vector in $\ell_{n+1}$. Then defining $e_i = -a_i e_i'$, in the basis $\{e_i\}$ the arrangement $\{\ell_1,\dots,\ell_{n+1}\}$ is visibly isomorphic to the standard construction.

Continuing to work in these coordinates (where in particular $v = -e_1 - \dots - e_n$), we next posit the existence of an additional line $\ell_{n+2}$, spanned by some element 
\[
w = \sum_{i = 1}^n b_i e_i.
\]
We consider the element $P$ realizing the transposition $(n+1,n+2)$. As this fixes $\ell_i$ for $i \le n$, necessarily $P$ is diagonal. As $P(\ell_{n+1}) = \ell_{n+2}$, we can assume that the diagonal entries are given by the coefficients $b_i$, by adjusting $P$ by some scalar matrix if necessary. As also $P(\ell_{n+2}) = \ell_{n+1}$, it follows that $b_i^2 = 1$ for $i = 1, \dots, n$; without loss of generality we set $b_1 = 1$.

We next examine the element $Q \in \GL_n(\C)$ acting as the transposition $(1, n+1)$. As this fixes the lines $\ell_j$ for $j \ne 1$ and takes $\ell_1$ to $\ell_{n+1}$, this is given in coordinates as
\[
Q(e_j) = \begin{cases}
                -c_{1}\sum_{i = 1}^n e_i & j = 1\\
                c_{j} e_j & j \ne 1
\end{cases}
\]
for some scalars $c_{j} \in \C^{\times}$. We must also have $Q(\ell_{n+1}) = \ell_1$, so that $Q(v)=\lambda e_1$ for some $\lambda \ne 0$. Examining the $e_i$-coefficient of $Q(v)$ for $i \ne 1$ shows that $c_i = c_1 :=c$ is independent of $i$. 

We have $Q(\ell_{n+2}) = \ell_{n+2}$, so that $w$ is an eigenvalue of $Q$ with some eigenvalue $\lambda$. From above, recalling $b_1 = 1$,
\begin{align*}
    Q(w) &= Q(\sum_{i=1}^n b_i e_i)\\
        &= \sum_{i = 1}^n b_i Q(e_i)\\
        &= -c e_1 + \sum_{i = 2}^n c(b_i-1) e_i.
\end{align*}
Examining the first component of this shows that $\lambda = -c$. Examining the second component then shows that $-c b_2 = c(b_2-1)$, so that $b_2 = 1/2$, in contradiction with the condition $b_2^2 = 1$. (We note that this argument would break down over a field of characteristic $3$).

The second statement, concerning the maximality and uniqueness of the dual simplex arrangement, follows directly from the above after passing to the dual space.
\end{proof}

\para{The stabilizer of the simplex arrangement} Recall from \Cref{defintion:arrstab} that the {\em stabilizer} of an arrangement $\mathcal W \subset \C^n$ consists of those matrices $A \in \GL_n(\C)$ that fix each subspace $W_i \in \mathcal W$. Any stabilizer $\Stab(\mathcal W)$ contains the group $\C^\times I$ of scalar matrices. Here we show that for the simplex arrangement and its dual, this is the {\em entire} stabilizer.

\begin{lemma}\label{lemma:stabstd}
Let $\mathcal W \subset \C^n$ denote either the simplex arrangement $\mathcal S_n$ or its dual $\mathcal S_n^*$. Then $\Stab(\mathcal W) = \C^\times I$.
\end{lemma}
\begin{proof}
In the standard coordinates $\{e_1, \dots, e_n\}$ on $\C^n = V_{std}^n$, the simplex arrangement $\mathcal S_n$ has $\ell_i = \pair{e_i}$ for $1 \le i \le n$ and $\ell_{n+1} = \pair{e_1+ \dots + e_n}$. Any $A \in \GL_n(\C)$ that fixes $\ell_1, \dots, \ell_n$ is therefore diagonal; fixing $\ell_{n+1}$ forces each diagonal entry to be equal, showing that $A = \lambda I$ is scalar. As the stabilizer of a dual arrangement is evidently the dual of the stabilizer, the claim follows for $\mathcal S_n^*$ as well.
\end{proof}

\begin{corollary}
\label{lemma:mustbestd}
Let $\mathcal W \subset V_{std}^n$ be the simplex arrangement $\mathcal S_n$ or its dual $\mathcal S_n^*$, viewed as a hyperplane arrangement in $V_{std}^n$. Let $\sigma \mapsto P_\sigma \in \GL(V_{std}^n)$ be the standard representation of $\Sigma_{n+1}$. Then every realization map $\rho: \Sigma_{n+1} \to \GL(V_{std}^n)$ for $\mathcal W$ is of the following form:
\[
\rho(\sigma) = \lambda_\sigma P_\sigma
\]
for some set of $\{\lambda_\sigma \mid \sigma \in \Sigma_{n+1}\} \subset \C^\times$. In other words, there is a unique projective class of realization map, given by the standard representation.
\end{corollary}

\begin{proof}
The standard representation $\rho_{std}:\Sigma_{n+1} \to \GL(V_{std}^n)$ is a realization map. If $\rho: \Sigma_{n+1} \to \GL(V_{std}^n)$ is any other realization map, then the function $\sigma \mapsto \rho(\sigma)^{-1} \rho_{std}(\sigma)$ is valued in the stabilizer $\Stab(\mathcal W)= \C^\times I$, from which the claim follows.
\end{proof}

\begin{remark}
There is an evident similarity between \Cref{lemma:stabstd} and Schur's lemma, in light of the fact that $V_{std}^n$ is irreducible. However it is {\em not} the case that if $V$ is an irreducible representation, then any system of subrepresentations necessarily has stabilizer $\C^\times I$ - e.g. this fails if one takes $W_1 = V$ to be the entire space. It is not immediately clear to the authors if any {\em nondegenerate} system of subrepresentations obtained from an irrep $V$ necessarily has stabilizer $\C^\times I$.
\end{remark}

\para{The $\mathbf{\tilde \Sigma_5}$ arrangement} There is one additional sporadic arrangement that will play an important role in the paper. It arises via a {\em projective representation} of $\Sigma_5$, i.e. a homomorphism $\rho: \Sigma_5 \to \PGL(V)$. For the sake of avoiding a lengthy digression into the theory, we will postpone detailed calculations to \Cref{appendix:s5} and present only an overview in the body of the paper.

\begin{example}[The $\tilde \Sigma_5$-arrangement]\label{example:s5tilde}
Let $V^5_{basic} \cong \C^4$ denote the {\em basic representation} of $\Sigma_5$, a projective representation of dimension $4$. By general theory, $V^5_{basic}$ is a linear representation of a $\Z/2\Z$ extension $\tilde \Sigma_5$ of $\Sigma_5$ (indeed, there are two non-isomorphic such groups $\tilde \Sigma_5$ and $\hat \Sigma_5$, but both induce the same projective representation of $\Sigma_5$).

We apply the system of subrepresentations construction (\Cref{subrepsystem}) to $\tilde \Sigma_5$ acting on $V^5_{basic}$; for simplicity's sake, we formulate the discussion in the setting of projective representations of $\Sigma_5$.

Embed $\Sigma_4$ into $\Sigma_5$ as the stabilizer of $1 \in [5]$. Then $V^5_{basic}$ decomposes as a projective $\Sigma_4$-representation as follows:
\[
V^5_{basic} = V_{basic}^{4} \oplus V_{basic}^{4,a},
\]
with $V_{basic}^{4},V_{basic}^{4,a}$ the pair of ``associate'' basic representations of $\Sigma_4$, each of dimension $2$. This can be seen via a character computation - see the character tables in \cite[pp. 43, 80]{HH}. The $\tilde \Sigma_5$ arrangement $\mathcal W_{\tilde \Sigma_5}$ is then defined as the system of subrepresentations for $V_{basic}^4 \le V_{basic}^5$. In \Cref{example:s5tildeexplicit} in \Cref{appendix:s5}, we give an explicit set $\mathcal W = \{W_1, \dots, W_5\}$ of $4 \times 2$ matrices spanning the associated subspaces; here we remark that this arrangement has the property that $W_i \oplus W_j = \C^4$ for any pair of distinct indices $i \ne j$.
\end{example}

\subsection{The suspension construction} One can use a totally symmetric arrangement to produce interesting examples of reducible totally symmetric sets of endomorphisms, via a procedure called the {\em suspension construction}.

\begin{construction}[Suspension]
\label{construction:suspension}
Let $\mathcal W = \{W_1, \dots, W_k\}$ be a set of $d$-planes in $\C^n$. Suppose that $\mathcal W$ satisfies a strong form of total symmetry: there exists a realization map $\sigma \mapsto P_{\sigma} \in \GL_n(\C)$ and a set of coset representatives $\{M_i\} \subset M_{n,d}(\C)$ for $\{W_i\}$ such that $P_\sigma M_i = M_{\sigma(i)}$. The {\em suspension} $\Sigma \mathcal W$ of $\mathcal W$ is the $k$-element commutative totally symmetric set in $M_{n+d}(\C)$ given by
\[
\Sigma W_i = \abcd{\lambda I}{M_{i}}{0}{\lambda I}
\]
for some choice of parameter $\lambda \in \C$. It is clear that the elements of $\Sigma \mathcal W$ commute. To see that $\Sigma \mathcal W$ is totally symmetric, we observe that under the standard embedding $\C^n \into \C^{n+d}$, the action of $P_\sigma$ extends to an automorphism $P_\sigma \oplus I_d$ of $\C^{n+d}$. Then it is easy to check that
\[
(P_\sigma \oplus I_d) \abcd{\lambda I}{M_{i}}{0}{\lambda I} (P_\sigma \oplus I_d)^{-1} = \abcd{\lambda I}{P_\sigma M_{i}}{0}{\lambda I} = \abcd{\lambda I}{M_{\sigma(i)}}{0}{\lambda I},
\]
demonstrating total symmetry.
\end{construction}

A particularly important example of suspension is given by inputting the simplex arrangement $\mathcal S_n$ of \Cref{example:simplex}.

\begin{example}[The simplex construction]
\label{example:simplexconst}
One observes that the lines of the simplex arrangement $\mathcal S_n$ satisfy the strong form of total symmetry required in \Cref{construction:suspension}. We call the suspension $\Sigma \mathcal S_n$ the {\em simplex construction}. E.g. for $n = 3$, this produces the following $4$-element totally symmetric set in $\End(\C^4)$:

\begin{eqnarray*}
&\Sigma W_1 = \left( \begin{array}{ccc|c}\lambda &&& 1\\ & \lambda && 0\\ &&\lambda&0\\ \hline &&& \lambda \end{array} \right), &\Sigma W_2 = \left( \begin{array}{ccc|c}\lambda &&& 0\\ & \lambda && 1\\ &&\lambda&0\\ \hline &&& \lambda \end{array} \right)\\
&\Sigma W_3 = \left( \begin{array}{ccc|c}\lambda &&& 0\\ & \lambda && 0\\ &&\lambda&1\\ \hline &&& \lambda \end{array} \right), &\Sigma W_4 = \left( \begin{array}{ccc|c}\lambda &&& -1\\ & \lambda && -1\\ &&\lambda&-1\\ \hline &&& \lambda \end{array} \right)
\end{eqnarray*}
In the case $n = 1$, the simplex arrangement itself is degenerate, but the pair of matrices
\[
A_1 = \abcd{\lambda}{1}{0}{\lambda}, \qquad A_2 = \abcd{\lambda}{-1}{0}{\lambda}
\]
are still totally symmetric; we consider this as arising from the simplex construction as well.
\end{example}

\subsection{Decomposition systems and the eigenspace construction} There is a second method for passing from a totally symmetric arrangement to a totally symmetric set that in general produces non-commutative such sets. This takes as input a {\em decomposition system} for a vector space $V$, which we now define.

\begin{definition}[Decomposition system]
A {\em decomposition system $\mathcal D$ of size $k$} for a vector space $V$ is a collection $W_{i,j}$ (with $i \in [k]$ and $j \in [p]$) of subspaces of $V$ such that
\begin{enumerate}
    \item For fixed $i$, there is a direct sum decomposition
    \[
    V = \bigoplus_{j \in [p]} W_{i,j}.
    \]
    \item The decompositions $\bigoplus_{j \in [p]} W_{i,j}$ are totally symmetric in the following sense: for every $\sigma \in \Sigma_{[k]}$, there is $P_\sigma \in \GL(V)$ such that 
    \[
    P_\sigma W_{i,j} = W_{\sigma(i),j}.
    \]
\end{enumerate}
Decomposition systems arise naturally from systems of subrepresentations: one simply fixes not just a $G_1$-subrepresentation $W_1 \le V$, but a decomposition $V = \bigoplus W_j$ as a $G_1$-representation, and then proceeds as in \Cref{subrepsystem} to promote this to a decomposition system. 
\end{definition}

\begin{example}[The simplex system]\label{example:simplexsystem}
The {\em simplex system} $\mathcal D_{simp}^n$ is the decomposition system for the standard representation: for each $i\in [n+1]$, the ambient space $V_{std}^n$ decomposes as 
\[
V_{std}^{n} = (V_{std}^{n-1})_i \oplus (\C)_i.
\]
It is totally symmetric of size $n+1$.
\end{example}

\begin{example}[The $\tilde \Sigma_5$ system]\label{example:s5system}
The second example of a decomposition system we will need arises from the $\tilde \Sigma_5$ arrangement. Returning to the discussion in \Cref{example:s5tilde}, for {\em each} subgroup $\Stab(i) \cong \Sigma_4$ of $\Sigma_5$, one has a canonical decomposition as a projective $\Stab(i)$-representation
\[
V^5_{basic} = (V_{basic}^{4})_i \oplus (V_{basic}^{4,a})_i,
\]
which therefore determines a decomposition system $\mathcal D_{\tilde \Sigma_5}$ on $V_{basic}^5 \cong \C^4$.
\end{example}

Given a decomposition system, one can use a method called the {\em eigenspace construction} to build an associated totally symmetric set.

\begin{construction}[The eigenspace construction]\label{espace}
Let $\mathcal D = \{W_{i,j}\}$ be a decomposition system for $V$. Choose distinct elements $\lambda_1, \dots, \lambda_p \in \C$, and then define a set
\[
\mathcal E(\mathcal D) = \{A_1, \dots, A_k\} \subset \End(V)
\]
via the condition that
\[
E_{\lambda_j}(A_i) = W_{i,j}.
\]
The total symmetry of the decomposition system then gives rise to the total symmetry of $\mathcal E(\mathcal D)$.
\end{construction}

\begin{example}[Non-commutative simplex construction]
\label{example:ncsimplex}
Applying the eigenspace construction to the simplex system of \Cref{example:simplexsystem} gives the {\em non-commutative simplex construction} $\mathcal E (\mathcal D_{simp}^k)$, a $k$-element totally symmetric set in $\GL_{k-1}(\C)$. One has the following explicit expressions for $A_i$:
\begin{equation}\label{equation:ncsimplex}
A_i(v) = \mu v + \tfrac{\lambda-\mu}{n} \alpha_i(v) e_i.
\end{equation}
\end{example}

\begin{example}\label{ex:k3n2}
For later use, it will be convenient to have a realization by explicit matrices in the case $k=3$. An easy computation with the standard representation for $\Sigma_3$ gives the decompositions for $V_{std}^3$, leading to the description of $\mathcal E(\mathcal D_{simp}^3)$ as 
\[
A_1  = \abcd{\lambda}{\frac{\mu-\lambda}{2}}{0}{\mu} \qquad 
A_2  = \abcd{\mu}{0}{\frac{\mu-\lambda}{2}}{\lambda} \qquad
A_3  = \abcd{\frac{\lambda+\mu}{2}}{\frac{\lambda-\mu}{2}}{\frac{\lambda-\mu}{2}}{\frac{\lambda+\mu}{2}}.
\]
\end{example}

\begin{example}[$\tilde \Sigma_5$ construction]\label{s5construction}
Applying the eigenspace construction to the $\tilde \Sigma_5$ system leads to a five-element totally symmetric set
\[
\mathcal A_{\tilde \Sigma_5} \subset \GL_4(\C).
\]
A more explicit description of $\mathcal A_{\tilde \Sigma_5}$ is given in \Cref{example:s5constructionexplicit} in \Cref{appendix:s5}.
\end{example}

\section{The induction hypotheses and the base cases}\label{section:indstart}
We will prove the following results simultaneously by induction on the dimension $n$:

\para{\Cref{theorem:sizebound}} {\em Any totally symmetric set $\mathcal A \subset \End(\C^n)$ has cardinality $k \le n+1$. If $\mathcal A$ is moreover commutative, then $k \le n$.}

\para{\Cref{maintheorem:max}} {\em Let $\mathcal{A}\subset \End(\C^n)$ be a totally symmetric set (commutative or otherwise) of the maximal cardinality allowed by \Cref{theorem:sizebound}. Then

\begin{enumerate}
    \item If $\mathcal A$ is noncommutative and $n \ne 5$, then it arises via the ``noncommutative simplex construction'' of \Cref{example:ncsimplex}. 
    \item If $\mathcal A$ is noncommutative and $n = 5$, then it arises either via the noncommutative simplex construction or else is the ``$\tilde \Sigma_5$ construction'' of \Cref{s5construction}.
    \item If $\mathcal A$ is commutative and $n \ne 4$, then $\mathcal{A}$ is either the ``standard construction'' of \Cref{example:standard} or the ``simplex construction'' of \Cref{example:simplex}.
    \item If $\mathcal A$ is commutative and $n = 4$, then $\mathcal{A}$ is either standard, simplex, or the sporadic construction of \Cref{lemma:4}.
\end{enumerate}
}

\begin{theorem}\label{theorem:tsabound}
A nondegenerate totally symmetric arrangement $\mathcal W = \{W_1, \dots, W_k\} \subset \C^n$ has cardinality $k \le n+1$. Moreover, this inequality is strict unless $\mathcal W$ is the simplex arrangement or its dual, or else if $k = 5$ and $\mathcal W$ is the $\tilde \Sigma_5$-arrangement of \Cref{example:s5tilde}.
\end{theorem}

These results are intertwined with each other - to classify totally symmetric arrangements, one must understand a classification of (smaller-dimensional) totally symmetric sets, and vice versa. This necessitates a single induction with multiple hypotheses. We enumerate these statements below. As a mnemonic, {\bf S} indicates a statement about totally symmetric {\bf S}ets, (with {\bf cS} denoting a statement about {\bf c}ommutative totally symmetric sets), {\bf A} indicates a statement about totally symmetric {\bf A}rrangements, {\bf B} indicates a statement about {\bf B}ounds, and {\bf C} indicates a {\bf C}lassification statement. Note also that the indexing on the {\bf A} statements is ahead by one relative to the {\bf S} statements.

\begin{itemize}
    \item \AB{n}: Let $\mathcal W = \{W_1, \dots, W_k\} \subset \C^{n+1}$ be a nondegenerate totally symmetric arrangement. Then $k \le n+2$.
    \item \AC{n}: Let $\mathcal W = \{W_1, \dots, W_{n+2}\} \subset \C^{n+1}$ be a nondegenerate totally symmetric arrangement of size $k = n+2$. Then $\mathcal W$ is either the simplex arrangement or its dual, or else $n+1 = 5$ and $\mathcal W$ is the $\tilde \Sigma_5$ arrangement of \Cref{example:s5tilde}.
    \item \SB{n}: Let $\mathcal A = \{A_1,\dots, A_k\}\subset \GL_n(\C)$ be a totally symmetric set. Then $k \le n+1$, and if $\mathcal A$ is commutative, then $k \le n$.
    \item \SC{n}: Let $\mathcal A = \{A_1,\dots, A_{n+1}\}\subset \GL_n(\C)$ be a totally symmetric set of size $k = n+1$. Then $\mathcal A$ arises from a decomposition system associated to one of the maximal arrangements described in \AC{n-1}.
    \item \cSC{n}: Let $\mathcal A = \{A_1,\dots, A_{n}\}\subset \GL_n(\C)$ be a commutative totally symmetric set of size $k = n$. Then $\mathcal A$ arises via the standard construction, the simplex construction, or else $n=4$ and $\mathcal A$ is the sporadic example of \Cref{lemma:4}.
    \item \IH{n}: The {\em total induction hypothesis} \IH{n} consists of all of the above statements for all $m \le n$.
\end{itemize}

The statements \AB{n} - \cSC{n} are addressed over the course of \Cref{section:sb,section:sc,section:csc,section:abac} (as \Cref{prop:sbn,prop:sc,prop:csc,prop:ab,prop:ac}, respectively), together completing the inductive step of the argument. Here we treat the base case $n = 1$.

\begin{proposition}\label{prop:base}
The statements \AB{1}, \AC{1}, \SB{1}, \SC{1}, \cSC{1} hold.
\end{proposition}
\begin{proof}
The statements \AB{1} and \AC{1} concern totally symmetric arrangements in $\C^2$. Such arrangements are necessarily arrangements of lines, and so \Cref{proposition:simpmax} in the case $n = 2$ specializes to give \AB{1} and \AC{1}. For \SB{1}, as $\GL_1(\C)$ is abelian, any totally symmetric set is a singleton. The statement \SC{1} is vacuous. The statement \cSC{1} also holds: the singleton totally symmetric set can be viewed as arising from the standard construction for $n = 1$. 
\end{proof}

To finish this section, we verify that the base case \Cref{prop:base} together with the inductive steps \Cref{prop:sbn,prop:sc,prop:csc,prop:ab,prop:ac} do indeed prove \Cref{theorem:sizebound}, \Cref{maintheorem:max}, and \Cref{theorem:tsabound}, packaged together as the assertion \IH{n}.

\begin{proof}[Proof of \IH{n}]
We proceed by induction, the case $n = 1$ holding from \Cref{prop:base}. Assuming \IH{n-1}, \Cref{prop:sbn} establishes \SB{n}. The assertion \SC{n} then follows from \IH{n-1}, \SB{n}, and \Cref{prop:sc}, and \cSC{n} follows from \IH{n-1} and \Cref{prop:csc}. Finally, \AB{n} and \AC{n} follow from \IH{n-1}, \SB{n}, \SC{n}, and \Cref{prop:ab}, resp. \Cref{prop:ac}.
\end{proof}

\section{{\bf AB} and {\bf AC}: Maximal totally symmetric arrangements}\label{section:abac}

We first address the statements \AB{n} and \AC{n} concerning bounds and classifications of totally symmetric arrangements. The basic strategy is as follows: if $\mathcal W = \{W_1, \dots, W_k\}$ is a totally symmetric arrangement and $W_i \cap W_j$ is positive-dimensional, then $\{W_1 \cap W_i\}$ is a totally symmetric arrangement in $W_1$, to which induction can be applied; in particular, this applies whenever $2 \dim(W_i) > n$. If $2 \dim(W_i) < n$, then one can pass to the dual arrangement and apply this technique there. This leaves the case where $W_i \oplus W_j = \C^n$ as an exception. We will require separate techniques in this setting; we begin the section with an analysis of this situation.

\subsection{Half-dimensional arrangements}
Here we undertake a study of arrangements $\mathcal W$ with the property that $W_i \oplus W_j = \C^n$. The major objectives are \Cref{lemma:halfdim,lemma:d2,lemma:halfdimd3}, which give bounds on the sizes of such arrangements - \Cref{lemma:halfdim} establishes a general bound, and \Cref{lemma:d2} and \Cref{lemma:halfdimd3} improve this bound in the low dimensions $n = 4, 6$, respectively. In the case $n = 4$ we see that the only such arrangement is the $\tilde \Sigma_5$ arrangement.
\begin{lemma}
\label{lemma:coords}
Let $W_1, \dots, W_k \subset \C^{2d}$ be $d$-planes such that $W_i \oplus W_j = \C^{2d}$ for any pair of distinct indices $i\ne j$. Then there exist coordinates on $\C^{2d}$ in which $\{W_i\}$ admit $\GL_d(\C)$-coset representatives of the following form:
\[
W_1 = \binom{I}{0}, \qquad W_2 = \binom{0}{I}, \qquad W_3 = \binom{I}{I}, \qquad W_i = \binom{A_i}{I}\ (i \ge 4)
\]
with $A_i$ invertible for $i \ge 4$ (each indicated block is of size $d \times d$). Moreover, for any $M \in \GL_d(\C)$, the action of $M \oplus M$ on $\{W_i\} \subset \C^{2d}$ admits representatives with $W_1,W_2,W_3$ as above and with $A_i$ replaced with $MA_iM^{-1}$ for $i \ge 4$.
\end{lemma}
\begin{proof}
Choose a basis $\beta_1$ for $W_1$ and $\beta_2$ for $W_2$. As $W_1 \oplus W_2 = \C^{2d}$, the concatenation of $\beta_1$ and $\beta_2$ forms a basis for $\C^{2d}$. In any such basis, $W_1$ admits the coset representative $\binom{I}{0}$, and $W_2$ admits the representative $W_2 = \binom{0}{I}$. Choose a coset representative $\binom{A}{B}$ for $W_3$. Following the above discussion, since $W_1\cap W_3 = W_2 \cap W_3 = \{0\}$, each of $A, B$ is invertible. Then in the basis $A\beta_1 \cup B \beta_2$, each $W_i$ for $i \le 3$ admits the desired coset representatives.

The remaining elements $W_i$ for $i \ge 4$ admit coset representatives of the form $W_i = \binom{B_i}{C_i}$; again by the above discussion, each $B_i, C_i$ is invertible. Adjusting the coset representative by $C_i^{-1}$ and defining $A_i = B_i C_i^{-1}$ represents each $W_i$ in the required form.

To establish the final assertion, we observe that acting on $\C^{2d}$ via $M\oplus M$ and on $\C^{d}$ via $M^{-1}$ fixes the representatives for $W_1, W_2, W_3$ while sending $\binom{A_i}{I}$ to $\binom{M A_i M^{-1}}{I}$ for $i \ge 4$. 
\end{proof}

\begin{lemma}\label{lemma:halfdim}
Let $\mathcal W = \{W_1, \dots, W_k\}$ be a maximal totally symmetric arrangement of $d$-planes in $\C^{2d}$ with the property that $\C^{2d} = W_i \oplus W_j$ for any pair of distinct indices $i \ne j$. Then in the coordinates of \Cref{lemma:coords}, the matrices $A_i$ for $i \ge 4$ form a totally symmetric set in $\GL_d(\C)$ of size $k-3$. In particular, letting $M(d)$ denote the maximal cardinality of a (non-commutative) totally symmetric set in $\End(\C^d)$, then $k \le M(d) + 3$.
\end{lemma}
\begin{proof}
By \Cref{lemma:coords}, we can assume that $\{W_i\}$ admit coset representatives of the form
\[
W_1 = \binom{I}{0}, \qquad W_2 = \binom{0}{I}, \qquad W_3 = \binom{I}{I}, \qquad W_i = \binom{A_i}{I}\ (i \ge 4)
\]
with $A_i$ invertible. We will show that $\{A_i\}$ forms a totally symmetric set in $\GL_d(\C)$ of size $k - 3$.

To see this, let $\sigma \in \Sigma_k$ fix the elements $1,2,3$ pointwise. Total symmetry of $\{W_i\}$ implies the existence of $P_\sigma \in \GL(V)$ such that $W_{\sigma(i)} = P_\sigma W_i$ as $\GL_d(\C)$-cosets. For $\sigma$ fixing $1,2,3$, this causes $P_\sigma$ to be of the form
\[
P_\sigma = \begin{pmatrix} X_\sigma & \\ & X_\sigma \end{pmatrix}
\]
for some $X_\sigma \in \GL_d(\C)$. For $i \ge 4$, the equation $W_{\sigma(i)} = P_\sigma W_i$ implies that
\[
\binom{A_{\sigma(i)}}{I} = \binom{X_\sigma A_i}{X_\sigma} Y
\]
for some $Y \in \GL_d(\C)$. Consequently $Y = X_\sigma^{-1}$, showing that $\{A_i\mid i \ge 4\}$ is a totally symmetric set in $\GL_d(\C)$ of size $k-3$ as claimed. 
\end{proof}

\para{Exceptional cases; $\mathbf{d=2}$} \Cref{lemma:halfdim} establishes a bound on  the size of certain totally symmetric arrangements in $\C^{2d}$ in terms of the size of a totally symmetric set in $\End(\C^d)$. We will shortly see that for $d \ge 4$ this bound is sufficient to establish \AB{2d-1} and \AC{2d-1}, but some additional work is required to address the cases $d \le 3$. 

In \Cref{subsection:halfdim}, we will specialize the above analysis to the case of $d=2$, i.e. the setting of the $\tilde \Sigma_5$ arrangement. We will obtain the following results:

\para{\Cref{lemma:d2}} {\em Let $\mathcal W = \{W_1, \dots, W_k\}$ be a nondegenerate totally symmetric arrangement of $2$-planes in $\C^4$ with the property that $\C^4 = W_i \oplus W_j$ for any pair of distinct indices $i \ne j$. Assume the classification \SC{2} of totally symmetric sets of $3$ elements in $\GL_2(\C)$. Then $k \le 5$.}\\

\para{\Cref{lemma:d2unique}} {\em In the setting of \Cref{lemma:d2}, if $\mathcal W = \{W_1, \dots, W_5\}$ is of maximal size, then $\mathcal W$ is the $\tilde \Sigma_5$ arrangement.}\\

\para{\Cref{corollary:reals5}} {\em There is a unique projective class $\overline \rho: \Sigma_5 \to \GL_4(\C) \to \PGL_4(\C)$ of realization map for $\mathcal W_{\tilde \Sigma_5}$, realized by the basic projective representation of $\Sigma_5$.}\\

 \para{Half-dimensional arrangements - the case $\mathbf{d = 3}$} 
 Finally, we consider the other exceptional case $d=3$. We begin with an analysis of the eigenvalues of the totally symmetric set $\{A_i\mid i \ge 4\}$; this will also be used in \Cref{subsection:halfdim}.
 
 \begin{lemma}\label{lemma:evals}
Let $\mathcal W = \{W_1, \dots, W_k\}$ be a maximal totally symmetric arrangement of $d$-planes in $\C^{2d}$ with the property that $\C^{2d} = W_i \oplus W_j$ for any pair of distinct indices $i \ne j$, and let $\{A_4, \dots, A_k\} \subset \GL_d(\C)$ denote the associated totally symmetric set of \Cref{lemma:halfdim}. Let $\Lambda = \{\lambda_1, \dots, \lambda_d\}$ denote the set of eigenvalues of any $A_i$, counted with multiplicity. Then $\Lambda$ is invariant under the involutions $\alpha(\lambda) = \lambda^{-1}$ and $\beta(\lambda) = 1 - \lambda$.
\end{lemma}

\begin{proof}
We make the following two observations, from which the result follows immediately.

\para{Observation 1} {\em Each matrix $A_i$ is conjugate to its inverse.}

To see this, we consider the element $(12) \in \Sigma_k$ and a realization $P_{12} \in \GL_{2d}(\C)$. Considering the effect on $W_1, W_2, W_3$ shows that 
\[
P_{12} = \begin{pmatrix}  & X_{12}\\X_{12}&\end{pmatrix}
\]
for some $X_{12} \in \GL_{d}(\C)$. As this must fix $W_i$ for $i \ge 4$, it follows that there is $Y_{12,i} \in \GL_d(\C)$ such that
\[
\begin{pmatrix}  & X_{12}\\X_{12}&\end{pmatrix} \begin{pmatrix}A_i \\ I \end{pmatrix} = \begin{pmatrix}A_i Y_{12,i} \\ Y_{12,i} \end{pmatrix},
\]
which shows $X_{12} = A_i Y_{12,i}$ and $X_{12}A_i = Y_{12,i}$. Together these give the required conjugacy.

\para{Observation 2} {\em Each matrix $A_i$ is conjugate to $I - A_i$.}

For this, we consider the action of $(23) \in \Sigma_k$ via a realization $P_{23} \in \GL_{2d}(\C)$. Considering the effect on $W_1, W_2, W_3$ shows that 
\[
P_{23} = \begin{pmatrix} X_{23} & -X_{23}\\&-X_{23}\end{pmatrix}
\]
for some $X_{23} \in \GL_d(\C)$. As this must fix $W_i$ for $i \ge 4$, it follows that there is $Y_{23,i} \in \GL_d(\C)$ such that
\[
\begin{pmatrix} X_{23} & -X_{23}\\&-X_{23}\end{pmatrix} \begin{pmatrix}A_i \\ I \end{pmatrix} = \begin{pmatrix}A_i Y_{23,i} \\ Y_{23,i} \end{pmatrix}.
\]
This shows that $Y_{23,i} = -X_{23}$ and that $X_{23}(A_i-I) = A_i Y_{23,i}$, which combine to give the claim.
\end{proof}
 
 \begin{lemma}\label{lemma:halfdimd3}
 Let $\mathcal W = \{W_1, \dots, W_k\}$ be a nondegenerate totally symmetric arrangement of $3$-planes in $\C^6$ with the property that $\C^6 = W_i \oplus W_j$ for any pair of distinct indices $i \ne j$. Assume the classification \SC{3} of totally symmetric sets of $4$ elements in $\GL_3(\C)$. Then $k \le 6$.
 \end{lemma}
 \begin{proof}
 Following \Cref{lemma:coords,lemma:halfdim}, we can express $\mathcal W$ as follows:
 \[
 W_1 = \binom{I}{0}, \qquad W_2 = \binom{0}{I}, \qquad W_3 = \binom{I}{I}, \qquad W_i = \binom{A_i}{I}\ (i \ge 4)
 \]
 with $\mathcal A = \{A_4, \dots, A_k\}$ forming a $(k-3)$-element totally symmetric set in $\GL_3(\C)$. If $k = 7$, then the hypothesis \SC{3} implies that $\mathcal A$ arises via the noncommutative simplex construction; in particular, $\mathcal A$ has exactly two distinct eigenvalues. On the other hand, by \Cref{lemma:evals}, the $3$-element set $\Lambda$ of eigenvalues with multiplicity is invariant under the involutions $\alpha(\lambda) = \lambda^{-1}$ and $\beta(\lambda) = 1 - \lambda$. For $\Lambda$ to have odd cardinality, it must contain the unique fixed point of $\beta$, and so $\tfrac 1 2 \in \Lambda$. The orbit of $\tfrac 1 2$ under $\alpha, \beta$ is the three-element set $\{\tfrac 1 2, 2, -1\}$, and so $\Lambda$ must be this set, contrary to the condition that $\Lambda$ contain only two distinct elements. 
 \end{proof}

 \subsection{Proofs of AB and AC}
 Recall the statements of \AB{n} and \AC{n}:
 \begin{itemize}
    \item \AB{n}: Let $\mathcal W = \{W_1, \dots, W_k\} \subset \C^{n+1}$ be a nondegenerate totally symmetric arrangement. Then $k \le n+2$.
    \item \AC{n}: Let $\mathcal W = \{W_1, \dots, W_{n+2}\} \subset \C^{n+1}$ be a nondegenerate totally symmetric arrangement of size $k = n+2$. Then $\mathcal W$ is either the simplex arrangement or its dual, or else $n+1 = 5$ and $\mathcal W$ is the $\tilde \Sigma_5$ arrangement of \Cref{example:s5tilde}. 
\end{itemize}

\begin{proposition} \label{prop:ab}
Assuming \IH{n-1} and \SB{n}, then also \AB{n} holds.
\end{proposition}
\begin{proof}
Suppose that $\mathcal W = \{W_1, \dots , W_k\} \subset \C^{n+1}$ is a nondegenerate totally symmetric arrangement of $d$-planes; we will show that $k \le n+2$.

We first consider the case where $W_1 \cap W_2$ is positive-dimensional. Then restriction of the realization map shows that 
\[
\mathcal{W}|_{W_1} = \{W_1 \cap W_2, W_1 \cap W_3,\ldots, W_1 \cap W_k\}
\]
is totally symmetric in $W_1$. If $\mathcal{W}|_{W_1}$ is non-degenerate, \AB{\dim(W_1)-1} gives 
\[
k-1 \leq \dim(W_1)+1 \leq n+1,
\]
establishing \AB{n}.

In the case where $\mathcal{W}|_{W_1}$ is degenerate, let $Q$ be the common intersection $Q = W_1 \cap W_i$. By total symmetry, $Q = W_i \cap W_j$ for any pair of indices $i,j$. We construct the reduced arrangement $\mathcal W^{red} \subset \C^{n+1}/Q$ as in \Cref{definition:quotientarrangement}; by \Cref{lemma:nondegquot}, it is nondegenerate. Applying \AB{\dim(\C^{n+1}/Q)-1} to $\mathcal W^{red}$ gives $k \leq \dim(\C^{n+1}/Q) + 1 = n +2 - \dim(Q) < n + 2$, establishing \AB{n}. 


The same argument can be applied if $\dim(W_i) < n/2$, by passing to the dual arrangement $\mathcal W^*$, which guarantees that $\dim(W_i^* \cap W_j^*) > 0$. This leaves only the case where both $W_i \cap W_j$ and $W_i^* \cap W_j^*$ are zero-dimensional, which in particular implies that $n+1 = 2d$ and $W_i \oplus W_j = \C^{n+1}$ for any pair of indices $i,j$, i.e. the setting of \Cref{lemma:halfdim}. This implies that $k \le M(d)+3$, where $M(d)$ is the maximal size of a totally symmetric set in $\GL_d(\C)$. As $d = (n+1)/2 \le n$, we apply \SB{d} to conclude that $k \le d+4$. For $d \ge 3$, the inequality $d+4 \le 2d+1 = n+2$ holds, establishing \AB{n}. 

The case $n = 1$ corresponding to $d = 1$ was addressed in \Cref{prop:base} as a base case. It remains to analyze the case $n = 3$ corresponding to $d = 2$ - this is addressed by \Cref{lemma:d2}.
\end{proof}

\begin{proposition}\label{prop:ac}
Assuming \IH{n-1}, \SB{n}, and \SC{n}, then also \AC{n} holds.
\end{proposition}
\begin{proof}
We consider a nondegenerate arrangement $W_1, \dots, W_{n+2} \subset \C^{n+1}$ of the maximal size permitted by \AB{n}. If $W_i \cap W_j$ is positive-dimensional for $i \ne j$, then the arrangement $\mathcal W|_{W_1} = \{W_1 \cap W_i \mid 2 \le i \le n+2\}$ is of cardinality $n+1$ in the space $W_1$. The argument of \Cref{prop:ab} shows that $\mathcal W|_{W_1}$ is nondegenerate. By \AB{\dim(W_1)-1}, it follows that $n+1 \le \dim(W_1) + 1$, so that $\dim(W_1) = n$, and the arrangement $\mathcal W$ is thus an arrangement of hyperplanes. By \Cref{proposition:simpmax}, $\mathcal W$ is the dual simplex arrangement.

Passing to the dual arrangement, if $W_i^* \cap W_j^*$ is positive-dimensional for $i \ne j$, then the same argument shows that $\mathcal W^*$ is the dual simplex arrangement, i.e. $\mathcal W$ is the simplex arrangement. Finally, we consider the case where both $W_i \cap W_j$ and $W_i^* \cap W_j^*$ are zero-dimensional, i.e. the setting of \Cref{lemma:halfdim}. Recall that \Cref{lemma:halfdim} asserts that the size $k$ of a totally symmetric arrangement of $d$-planes in $\C^{2d}$ satisfying $W_i \cap W_j = \{0\}$ is at most $M(d) + 3$, where $M(d)$ is the maximal size of a totally symmetric set in $\GL_d(\C)$. Appealing to \SB{d}, it follows that $k \le d+4$. For $d \ge 4$ the inequality $d+4 < 2d+1$ is sharp, leaving only the cases $d = 1,2,3$ to consider separately. 

If $d = \dim(W_i) = 1$, then $\mathcal W$ is an arrangement of $3$ lines in $\C^2$ which \Cref{proposition:simpmax} shows must be the simplex arrangement. If $d = 2$ (so that $n = 3$) since we are assuming \SC{2}, then \Cref{lemma:d2unique} asserts that $\mathcal W$ is the $\tilde \Sigma_5$ arrangement. Finally if $d = 3$ (in which case $n = 5$) since we assume \SC{3}, we appeal to \Cref{lemma:halfdimd3} to conclude that $k \le 6$ as required.
\end{proof}

\section{{\bf SB}: Bounds for totally symmetric sets}\label{section:sb}
We come next to the statement \SB{n} concerning bounds on the size of totally symmetric sets:
\begin{itemize}
    \item \SB{n}: Let $\mathcal A = \{A_1,\dots, A_k\}\subset \GL_n(\C)$ be a totally symmetric set. Then $k \le n+1$, and if $\mathcal A$ is commutative, then $k \le n$.
\end{itemize}

\begin{proposition}\label{prop:sbn}
Assuming \IH{n-1}, then also \SB{n} holds.
\end{proposition}
\begin{proof}
We consider a totally symmetric set $\mathcal A = \{A_1, \dots, A_k\}$ in $\GL_n(\C)$; our objective is to show that $k \le n+1$. In broad outline, the proof proceeds by considering a suite of arrangements canonically attached to $\mathcal A$. If any of these arrangements are nondegenerate, one of the hypotheses in \IH{n-1} can be applied. If all of the arrangements considered are degenerate, we use this to construct a ``derived'' totally symmetric set on a proper quotient of $\C^n$, to which induction can be applied. For readability, we organize the proof into three phases, further splitting the latter two into the {\em construction} of the objects under study, and their {\em analysis}. 

As a preliminary remark, we observe that \Cref{corollary:permrigid} implies that if $\mathcal W$ is any totally symmetric arrangement associated to a totally symmetric set $\mathcal A \subset \End(\C^n)$, then either $\mathcal W$ is degenerate, or else $\abs{\mathcal A} = \abs{\mathcal W}$. In particular, the bounds on $\abs{\mathcal W}$ imposed by the hypotheses \AB{m} do impose bounds on $\abs{\mathcal A}$. 

\para{Phase 1: arrangements of generalized eigenspaces} Let $\lambda$ be an eigenvalue of $\mathcal A$, and consider the arrangements $E_{\lambda,c}^i$ of degree-$c$ generalized eigenspaces for $c \ge 1$. If any one of these arrangements is nondegenerate, the hypothesis \AB{n-1} shows that $k \le n+1$. 

If $\mathcal A$ is commutative, then each $A_i \in \mathcal A$ preserves each eigenspace $E_{\lambda,c}^j$, and so there is a containment 
\[
\mathcal A \subset \Stab(E_{\lambda,c}^i)
\]
of $\mathcal A$ in the stabilizer subgroup of any eigenspace arrangement $\{E_{\lambda,c}^i\}$. From \AC{n-1}, such $\{E_{\lambda,c}^i\}$ is either the simplex arrangement or its dual, or else the $\tilde \Sigma_5$-arrangement. By \Cref{lemma:stabstd} or \Cref{lemma:stabs5}, the stabilizer consists only of scalar matrices, an absurdity. This establishes \SB{n} in this case.

\para{Phase 2a: Kernel and image arrangements (construction)} We therefore assume that all arrangements $E_{\lambda,c}^i$ are degenerate (going forward, we will therefore omit the superscripts). If there are multiple eigenvalues $\lambda_1, \dots, \lambda_p$ for $p > 1$, then let $d_i$ be the minimal integer for which $E_{\lambda_i, d_i} = E_{\lambda_i,d_i+1}$. Each $E_{\lambda_i,d_i}$ is a proper subspace of $\C^n$ that is moreover $(\mathcal A, \rho)$-invariant. It follows that if the restriction of $\mathcal A$ to each $E_{\lambda_i,d_i}$ is degenerate, then $\mathcal A$ is itself degenerate. Thus some such restriction must be nondegenerate, showing that the strong form of the bound $k \le n$ in \SB{n} must hold.

We are therefore reduced to the case where there is one eigenvalue $\lambda$, the arrangements $E_{\lambda,c}$ are degenerate and form strict subspaces of $\C^n$ for $c < d$, and $E_{\lambda,d} = \C^n$. We pause to note that necessarily $d > 1$, since otherwise each $A_i$ would act identically as multiplication by $\lambda$, contradicting nondegeneracy of $\mathcal A$.

Each $E_{\lambda,c}$ is an $(\mathcal A, \rho)$-invariant subspace. If the restriction of $\mathcal A$ to any such subspace for $c < d$ is nondegenerate, or if the induced totally symmetric set on any quotient by any such subspace is nondegenerate, then $\mathcal A$ induces a totally symmetric set on a proper $m$-dimensional subspace, and by \SB{m} we obtain the strong form of the bound $k \le n$ required for \SB{n}.

We therefore can further assume that the restriction of $\mathcal A$ to $E_{\lambda, d-1}$ is degenerate, and likewise that the quotient ${\mathcal A}/{E_{\lambda,1}}$ is degenerate. The former assumption implies that $E_{\lambda,d-1}$ is contained in the kernel of $A_i - A_1$ for any $i > 1$, and the latter implies that the image of $A_i - A_1$ is contained in $E_{\lambda,1}$. Thus we can write 
\[
A_i = A_1 + T_i
\]
for some $T_i: \C^n \to \C^n$ admitting a factorization
\[
T_i: \C^n \to \C^n/E_{\lambda, d-1} \to E_{\lambda,1} \subset \C^n.
\]
Restricting the realization map to the stabilizer of $1 \in [k]$, we see that
\[
\mathcal A' = \{A_i - A_1 \mid i > 1\} = \{T_i \mid i > 1\}
\]
forms a (nondegenerate) totally symmetric set of cardinality $k-1$. 

\para{Phase 2b: Kernel and image arrangements (analysis)} We first consider the case where the arrangement $\ker T_i \subset \C^n$ is nondegenerate. Under this assumption, \Cref{lemma:nondegquot} asserts that the reduced arrangement $\ker T_i^{red} \subset \C^n/ E_{\lambda,d-1}$ is also nondegenerate. Applying \AB{\dim(\C^n/E_{\lambda,d-1})-1} to $\ker T_i^{red}$ shows that
\[
k-1 \le \dim(\C^n / E_{\lambda,d-1}) +1.
\]
To establish the strong bound $k \le n$ of \SB{n}, we must show that $\dim(\C^n/E_{\lambda, d-1}) \le n-2$. If $\dim(\C^n/E_{\lambda, d-1}) = 1$, then $\ker T_i^{red}$ is an arrangement in a one-dimensional space, hence a singleton: $k = 2 \le n$. We therefore assume $\dim(\C^n/E_{\lambda, d-1})\ge 2$. By the the Jordan inequality \eqref{equation:jordan2}, 
\[
\dim(\C^n/E_{\lambda, d-1}) \le \dim(\C^n)/d = n/d
\]
(recall that $E_{\lambda,d} = \C^n$). Because $d \ge 2$ and $\dim(\C^n/E_{\lambda, d-1}) \ge 2$, it follows that $n \ge 4$, and therefore, the inequality
\[
n/d \le n-2
\]
holds. Thus the strong form $k \le n$ of \SB{n} holds.

We consider next the case where $\im T_i \subset E_{\lambda,1}$ is nondegenerate. If $\dim(E_{\lambda,1}) \le n - 2$, then \AB{\dim(E_{\lambda,1})-1} implies that $k - 1 \le n-1$ as required for the strong bound. We then suppose that $\dim(E_{\lambda,1}) = n-1$. In this case, necessarily $d = 2$ and then $A_i$ for $i \ge 1$ can be represented in the form 
\[
A_i = \lambda I + X_i,
\]
where $X_i$ admits a factorization $X_i: \C^n \to \C^n/E_{\lambda,1} \to E_{\lambda,1} \subset \C^n$. Then the set 
\[
\mathcal X = \{A_i - \lambda I \mid i \ge 1\} = \{X_i \mid i \ge 1\}
\]
is also totally symmetric with respect to the same realization map as $\mathcal A$. As the arrangement $\im T_i = \im(X_i-X_1) \subset E_{\lambda,1}$ is assumed to be nondegenerate, the same is true of the arrangement $\im X_i$, and we conclude from \AB{\dim(E_{\lambda,1})-1} that $k \le \dim(E_{\lambda,1})+1 \le (n-1) + 1$ as required for \SB{n} in its strong form.

\para{Phase 3a: The derived totally symmetric set (construction)} We therefore assume that the arrangements $\ker T_i$ and $\im T_i$ are degenerate, so that each $T_i$ has the same kernel $\ker T \supseteq E_{\lambda, d-1}$ and the same image $\im T \subseteq E_{\lambda, 1}$. Note in particular that $\ker T$ and $\im T$ are invariant under the restriction of $\rho$ to the stabilizer of $A_1$. Each $T_i$ then admits a factorization as follows:
\[
\xymatrix{
T_i: \C^n \ar[r] & \C^n/\ker T \ar[r]^{\overline{T_i}} &\im T \ar[r]^{\subset} &\C^n,
}
\]
where $\overline{T_i}$ is an isomorphism. Since $\ker T$ and $\im T$ are invariant under the stabilizer of $A_1$, total symmetry of the set $\mathcal A'$ of $T_i$'s induces total symmetry of the set 
\[
\overline{\mathcal{A}'} = \{\overline{T_i} \mid i > 1\},
\]
in the sense that for any $\sigma \in \Sigma_{k-1}$, there are automorphisms $P_\sigma \in \GL(\C^n/\ker T)$ and $Q_\sigma \in \GL(\im T)$ (induced from the action of $\rho$ on $\C^n$) for which 
\begin{equation}\label{gentotally symmetric set}
Q_\sigma \overline{T_i} P_\sigma^{-1} = \overline{T_{\sigma(i)}}.
\end{equation}
We are led to consider the {\em derived totally symmetric set}
\[
\mathcal A^{der} = \{\overline{T_2}^{-1} \overline {T_i} \mid i > 2 \} \subset \End(\C^n/\ker T).
\]
By \eqref{gentotally symmetric set}, $\mathcal A^{der}$ is a totally symmetric set of cardinality $k-2$ under the realization map where $\sigma$ acts by conjugation by $P_\sigma$ for $\sigma \in \Sigma_k$ fixing $1$ and $2$. We remark that $\mathcal{A}^{der}$ is similar to (but technically distinct from) the notion of ``derived totally symmetric set'' $\mathcal{A}'$ as used in \cite{Kordek-Margalit}.

\para{Phase 3b: The derived totally symmetric set (analysis)} Applying \SB{\dim(\C^n/\ker T)} to $\mathcal A^{der}$ gives $k-2 \le \dim(\C^n/\ker T) + 1$. To obtain the strong bound in \SB{n}, we will show that $\dim(\C^n/\ker T) \le n - 3$. As $\ker T \supseteq E_{\lambda,d-1}$, it follows from the Jordan inequality \eqref{equation:jordan2} that
\[
\dim(\C^n/ \ker T) \le \dim(\C^n/E_{\lambda, d-1}) \le n/d.
\]
The inequality 
\[
n/d \le n-3
\]
holds for $n \ge \frac{3d}{d-1}$, i.e. for $n \ge 6$ if $d = 2$, for $n \ge 5$ if $d = 3$, and for $n \ge 4$ if $d \ge 4$. Note that $d \le n$, so that the result is proved in the case $d \ge 4$. This leaves a handful of sporadic cases to analyze: $2 \le n \le 5$ for $d = 2$ and $3 \le n \le 4$ for $d = 3$.

We begin with $d = 3$. If $n \le 4$ then the Jordan inequality \eqref{equation:jordan2} implies that necessarily $\dim(\C^n/\ker T) = 1$. Thus $\mathcal A^{der}$ is a totally symmetric set in a one-dimensional space, which must be a singleton. Thus $k-2= 1$, proving the results in this case.

We next consider $d = 2$. For $n \le 3$ we must also have $\dim(\C^n/\ker T) = 1$ so the argument of the previous paragraph applies. If $n = 4,5$ then $\dim(\C^n/\ker T) \le 2$. Assuming $\dim(\C^n/\ker T) = 2$, we conclude $k - 2 \le 3$, i.e. $k \le 5$, so there is nothing to show in the case $n = 5$. This leaves the case of $n = 4$ and $\dim(\C^n/\ker T) = 2$. Necessarily then $\ker T = E_{ \lambda,1}$ must be $2$-dimensional. We apply a variant of the  construction of the derived totally symmetric set: the elements $A_i$ must be of the form $A_i = \lambda I + X_i$, with $X_i: \C^n \to \C^n/E_{\lambda,1} \to E_{\lambda,1}$. Since the arrangement $E_{\lambda,1} = \ker X_i$ is assumed to be degenerate, it follows that the elements $\overline{X_i}: V/E_{\lambda,1} \to E_{\lambda,1}$ are injective, hence isomorphisms. Arguing as above, the set 
\[
\overline{X_1}^{-1}\overline{X_i} : \C^n/{E_{\lambda,1}} \to \C^n/{E_{\lambda,1}}
\]
is totally symmetric in the $2$-dimensional space $\C^n/E_{\lambda}$, hence has size $k - 1 \le 3$ as required for \SB{n} in its strong form to hold.
\end{proof}

\begin{remark}
\label{remark:kn1}
We record for later use the following observation which was seen in the course of the proof: {\em if $\mathcal A \subset \GL_n(\C)$ is a totally symmetric set of size $n+1$, then necessarily some arrangement $\{E_{\lambda,c}\}$ is nondegenerate.}
\end{remark}

\begin{remark}\label{remark:noncom}
Even when the original totally symmetric set $\mathcal A$ is commutative, the derived set $\mathcal A^{der}$ (when it can be constructed) is not necessarily commutative. Thus the analysis of commuting totally symmetric sets inevitably leads one to consider this more general setting.
\end{remark}

\section{{\bf SC}: Classification of maximal totally symmetric sets}\label{section:sc}
In this section we consider the statement \SC{n} concerning the classification of totally symmetric sets in $\GL_n(\C)$ of maximal size $k = n+1$:
\begin{itemize}
    \item \SC{n}: Let $\mathcal A = \{A_1,\dots, A_{n+1}\}\subset \GL_n(\C)$ be a totally symmetric set of size $k = n+1$. Then $\mathcal A$ arises from a decomposition system associated to one of the maximal arrangements described in \AC{n-1}.
\end{itemize}
Recall that \AC{n-1} asserts that a maximal arrangement is either the simplex arrangement or its dual, or if $n = 5$, the $\tilde \Sigma_5$ arrangement.

In \Cref{appendix:s5}, we establish the following technical lemma which will be used in the course of \Cref{prop:sc}.

\para{\Cref{lemma:halfdim1espc}} {\em Let $\mathcal W_{\tilde \Sigma_5} = \{W_1, \dots, W_5\}$ denote the $\tilde \Sigma_5$ arrangement. There is no totally symmetric set $\mathcal A = \{A_1,\dots,A_5\} \subset \GL_{4}(\C)$ for which $E_{\lambda,1}(A_i) = W_i$ and $E_{\lambda,2}(A_i) = \C^{4}$.}

\begin{proposition}\label{prop:sc}
Assuming \IH{n-1} and \SB{n}, then also \SC{n} holds.
\end{proposition}
\begin{proof}
Let $\mathcal A = \{A_1, \dots, A_{n+1}\} \subset \GL_n(\C)$ be a totally symmetric set of maximal size $k = n+1$ allowed by \SB{n}. Following \Cref{remark:kn1}, the strong bound $k \le n$ holds unless $\mathcal A$ admits a nondegenerate arrangement of the form $E_{\lambda,c}^i$. If $n \ne 5$ then by \AC{n-1} this must be simplex or its dual; when $n=5$ there is the additional possibility that $E_{\lambda,c}^i$ is the $\tilde \Sigma_5$ arrangement.

Our analysis will proceed by a study of the possible arrangements for the other eigenspaces of $\mathcal A$. The uniqueness of the projective class of realization map for both (dual) simplex and $\mathcal W_{\tilde \Sigma_5}$, as obtained in \Cref{lemma:mustbestd,corollary:reals5} imply that {\em any other eigenspace arrangement must also be a system of subrepresentations for the same representation}. In particular, since both $V_{std}^n, V_{basic}^5$ are irreducible, no arrangement can be degenerate, as such arrangement would have to arise from a proper subrepresentation of $V_{std}^n$ or  $V_{basic}^5$.  

We first consider the case where the arrangement is simplex or its dual. As the decomposition $V_{std}^n = V_{std}^{n-1} \oplus \C$ is {\em canonical} (being the decomposition into isotypic subspaces), it follows that there is at most one other proper eigenspace arrangement, and hence $\mathcal A$ has at most two distinct eigenvalues. 

We next claim that $\mathcal A$ must have {\em exactly} two distinct eigenvalues. Suppose to the contrary that $\mathcal A$ has a single eigenvalue $\lambda$. The arrangement $E_{\lambda,1}^i$ is either the simplex arrangement or its dual. Suppose first that it is the dual simplex arrangement; necessarily then $E_{\lambda,2}^i = \C^n$ for all $n$. Then $A_i \in \mathcal A$ has an expression of the form
\[
A_i(v) = \lambda v + \alpha_i(v) w_i
\]
for some some set of nonzero vectors $w_i \in E_{\lambda,1}^i = \ker(\alpha_i)$. 

The spans of the $w_i$ (i.e. the arrangement $\im A_i - \lambda I$) is a totally symmetric line arrangement, which must therefore be the simplex arrangement: up to scale, $w_i = e_i$. By construction, this arrangement must be contained in $E_{\lambda,1}^i$, i.e. the dual simplex arrangement. But the dual simplex arrangement does not contain the simplex arrangement (this can be seen from the explicit representation given in \Cref{example:simplex}).

It remains to consider the case where $E_{\lambda,1}^i$ is the simplex arrangement. Then $E_{\lambda,2}^i$ is either the dual simplex arrangement or $\C^n$. In the latter case, the Jordan inequalities \eqref{equation:jordan2} force $n = 2$, in which case an easy variant of \Cref{lemma:halfdim1espc} (adapted to the simplex arrangement $\mathcal S_2$) asserts that no such totally symmetric set can exist. In the former, we would have the simplex arrangement $E_{\lambda,1}^i$ contained in the dual simplex arrangement $E_{\lambda,2}^i$, again impossible.

Thus we find that $\mathcal A$ must have exactly two distinct eigenvalues, one with $E_{\lambda,1}^i$ given by the simplex arrangement, and the remainder with $E_{\mu,1}^i$ given by the dual simplex arrangement---this is exactly the characterization of the noncommutative simplex construction in terms of the eigenspace construction given in \Cref{example:ncsimplex}.\\

It remains to consider the case where $E_{\lambda,c}^i$ is the $\tilde \Sigma_5$ arrangement. Again, the decomposition $V_{basic}^5 = (V_{basic}^4)_i \oplus (V_{basic}^{4,a})_i$ is canonical, so that there is at most one other proper eigenspace arrangement, necessarily given by $(V_{basic}^{4,a})_i$. 

If this arises, then this shows that $\mathcal A$ arises via the eigenspace construction as applied to the decomposition system $(V_{basic}^4)_i \oplus (V_{basic}^{4,a})_i$ as desired. Otherwise, $\mathcal A$ has exactly one eigenvalue, with $E_{\lambda,1}^i$ given by the $\tilde \Sigma_5$ arrangement and $E_{\lambda,2}^i = \C^4$, but this is precluded by \Cref{lemma:halfdim1espc}.
\end{proof}

\section{{\bf cSC}: Classification of maximal commutative totally symmetric sets}\label{section:csc}

Finally we come to the statement \cSC{n} concerning the classification of $n$-element commutative totally symmetric sets in $\GL_n(\C)$:
\begin{itemize}
    \item \cSC{n}: Let $\mathcal A = \{A_1,\dots, A_{n}\}\subset \GL_n(\C)$ be a commutative totally symmetric set of size $k = n$. Then $\mathcal A$ arises via the standard construction, the simplex construction, or else $n=4$ and $\mathcal A$ is the sporadic example of \Cref{lemma:4}.
\end{itemize}

\begin{proposition}\label{prop:csc}
Assuming \IH{n-1}, then also \cSC{n} holds.
\end{proposition}
\begin{proof}
If $\mathcal A$ is irreducible and satisfies $k = n$, then \Cref{corollary:partition} asserts that $\mathcal A$ arises via the standard construction.

We therefore assume that $\mathcal A$ is reducible. If there is any proper $(\mathcal A, \rho)$-invariant subspace $W \le \C^n$ for which the restriction or quotient of $\mathcal A$ is non-degenerate, then \SB{\dim(W)} asserts that $\abs{\mathcal A}$ is further bounded by the dimension of this sub- or quotient space, so that $k = n$ cannot hold. 

Since $\mathcal A$ is assumed to be reducible, there is {\em some} proper $(\mathcal A, \rho)$-invariant subspace $W \le \C^n$. By hypothesis, the restriction $\mathcal A|_W$ is degenerate, say given by the common matrix $A \in \End(W)$. Such $A$ admits some eigenvalue $\lambda$, so that the $n$-fold eigenspace $E_{\lambda} := E(A)_{\lambda,1}^{[n]}$ (in the notation of \Cref{definition:pfold}) is nonempty (note however that $E_\lambda$ need not be the entire $\lambda$-eigenspace of any $A_i$). 

By hypothesis, the quotient ${\mathcal A}/{E_\lambda}$ is also degenerate, say given by the common matrix $B \in \End(\C^n/E_{\lambda})$.  Choosing coordinates so as to render $B$ in Jordan form, we find that $\mathcal A$ has the following block structure, where $J_{\lambda_j}(B)$ denotes the Jordan block for the eigenvalue $\lambda_j$ of $B$:
\[
A_i = \left (\begin{array}{c|ccc}
        \lambda I   & X_{i,1}           & \dots     & X_{i,p} \\ \hline
                    &J_{\lambda_1}(B)   &           &   \\
                    &                   & \ddots    &    \\
                    &       &               & J_{\lambda_p}(B)            
\end{array}\right)
\]
Since $\mathcal A$ is commutative, the coordinates corresponding to the Jordan block for $\lambda_j$ can further be adjusted so as to set $X_{i,j}$ to zero for every eigenvalue $\lambda_j \ne \lambda$. Thus if {\em any} $\lambda_j \ne \lambda$, then $\mathcal A$ admits $E_{\lambda_j}$ as an $(\mathcal A, \rho)$-invariant subspace. Our assumptions force the quotient ${\mathcal A}/{E_{\lambda_j}}$ to be degenerate; in the coordinates above, this forces every $X_{i,j}$ to be constant, showing degeneracy of $\mathcal A$ itself.

We therefore assume that $A_i$ is of the form
\[
A_i = \left (\begin{array}{c|c}
        \lambda I   & X_{i} \\ \hline
                    &J_{\lambda}     
\end{array}\right)
\]
for $J_\lambda$ in Jordan canonical form. That is, there is a filtration $E_\lambda = W_1 \le W_2 \le \dots \le W_d = V$ for which, in a basis for $\C^n$ consisting of an increasing union of bases for $W_i$,
\begin{equation}\label{equation:ai}
A_i = \left (\begin{array}{c|cccccc}
        \lambda I   & X_{i,2}   &           &\dots      &           &X_{i,d}  \\ \hline
                    &\lambda I  & Y_3       &                           \\
                    &           & \lambda I & Y_4       &               \\
                    &           &           &\lambda I  &               \\
                    &           &           &           & \ddots &        \\
                    &           &           &           & &Y_d           \\
                    &           &           &           & &\lambda I
\end{array}\right),
\end{equation}
with each $Y_i$ having linearly-independent columns. 

Consider the arrangement $E_{\lambda,1}^i \le W_2 \le \C^n$. If this is nondegenerate, then $\ker X_{i,2}$ forms a nondegenerate arrangement of cardinality $k=n$ in $W_2/W_1$. As $\dim(W_2) \le n$ and $\dim(W_1) \ge 1$, it follows that in this case, $\dim(W_2) = n$ (i.e. $W_2 = \C^n$) and $\dim(W_1) = 1$. Then $A_i$ is visibly given as the suspension of some totally symmetric arrangement of $n$ dual vectors $X_{i,2}$ in the $n-1$-dimensional space $\C^n/W_1$, which must be the dual simplex arrangement by \Cref{proposition:simpmax}.

If instead $E_{\lambda, 1}^i$ is degenerate, we consider the restriction and quotient ${\mathcal A}|_{E_{\lambda,1}}$ and ${\mathcal A}/{E_{\lambda,1}}$. By hypothesis these are both degenerate; repeating the analysis above, we find matrices for $A_i$ of the form \eqref{equation:ai} but now with each $X_{i,2}$ having linearly-independent columns in addition to all of the $Y_j$ having such. It follows that the filtration $W_1 = E_{\lambda,1} \le \dots \le W_d = \C^n$ is in fact the filtration $\{E_{\lambda,c}\}$ of generalized eigenspaces, all of which form degenerate arrangements. By hypothesis the restriction and quotient of $\mathcal A$ on each $E_{\lambda,c}$ for $c < d$ is degenerate, and so $X_{i,c}$ is independent of $i$ for $c < d$.

Adjusting coordinates as in the Jordan decomposition, one can arrange to set $X_{i,c} = 0$ for $2 < c < d$. Defining $Y_2:= X_{i,2}$ in the case $d > 2$, we arrive at the following matrices for $A_i$: for $d > 3$,
\[
A_i = \left (\begin{array}{c|cccccc}
        \lambda I   & Y_2       &   0       &\dots      &   0       &X_{i,d}  \\ \hline
                    &\lambda I  & Y_3       &                           \\
                    &           & \lambda I & Y_4       &               \\
                    &           &           &\lambda I  &               \\
                    &           &           &           & \ddots &        \\
                    &           &           &           & &Y_d           \\
                    &           &           &           & &\lambda I
\end{array}\right),
\]
and for $d = 2$:
\[
A_i = \begin{pmatrix}
\lambda I   & X_{i,2}\\
            & \lambda I
\end{pmatrix}.
\]
The columns of $X_{i,2}$ and $Y_2$ are moreover linearly independent.

Consider first the case $d > 2$. We see that $\{A_i - \lambda I\}$ forms a totally symmetric set of $n$ elements. We examine the associated arrangement $\{\im(A_i - \lambda I) \} \le E_{\lambda,d-1}$ of their images. As $\dim(E_{\lambda, d-1}) \le n-1$, it follows from \AB{\dim(E_{\lambda, d-1})-1} that necessarily this must be an equality. Note that each space $\im(A_i - \lambda I)$ contains the image of $Y_2$, so that this arrangement is not {\em reduced} in the sense of \Cref{definition:quotientarrangement}. This contradicts the classification \AC{\dim(E_{\lambda, d-1})-1} of maximal arrangements in $E_{\lambda,d-1}$: all such arrangements are reduced.

Consider finally the case $d = 2$. The matrices $A_i - \lambda I$ form a totally symmetric set of cardinality $n$, showing that $\im X_{i,2}$ forms a totally symmetric arrangement in $E_{\lambda,1}$ of cardinality $n$. If this is nondegenerate, then \AB{\dim(E_{\lambda,1})-1} implies that $\dim(E_{\lambda,1}) = n-1$. Thus $\im X_{i,2}$ is an arrangement of {\em lines} and hence forms the simplex arrangement by \AC{\dim(E_{\lambda,1})-1}, showing that $\mathcal A$ is the suspension of the simplex arrangement as claimed.

If $\im X_{i,2}$ is degenerate, then arguing as in Phase 3a of the proof of \Cref{prop:sbn}, we quotient by this invariant subspace; in practice, we may assume that the matrices $\{X_{i,2}\}$ are square and invertible. As in Phase 3b, the set $\{X_{1,2}^{-1} X_{i,2}\}$ is totally symmetric of cardinality $n-1$, acting on a space of dimension at most $n/2$. By \SB{n/2}, it follows that $n-1 \le n/2+1$, which holds only for $n \le 4$ (note also that $n$ must be even). 

If $n = 2$, then we must study totally symmetric sets of $2$ elements of the form
\[
A_i = \begin{pmatrix} \lambda &x_i\\ 0 & \lambda \end{pmatrix}.
\]
with $x_i \ne 0$. An easy calculation shows that up to conjugation, there is exactly one, with $x_1 = 1$ and $x_2 = -1$; this arises from the simplex construction as claimed.

Finally we must analyze the case $n = 4$. Here we are interested in $4$-element totally symmetric sets in $\End(\C^4)$ of the form
\[
A_i = \begin{pmatrix} \lambda I & X_i\\ & \lambda I\end{pmatrix},
\]
where each $X_i \in \End(\C^2)$ is invertible. \Cref{lemma:4} shows that surprisingly, there is a unique sporadic example of this form.
\end{proof}

In preparation for \Cref{lemma:4}, we have the following straightforward result which can be checked by direct computation.
\begin{lemma}\label{lemma:conj}
Let $A_X \in \End(\C^{2n})$ be a block matrix of the form
\[
A_X = \abcd{\lambda I}{X}{0}{\lambda I}
\]
with each block of size $n \times n$ and $X \in \GL_n(\C)$. Suppose $M \in \GL_{2n}(\C)$ conjugates $A_X$ to some other matrix of the form $A_Y$. Then $M$ admits the structure of a block matrix
\[
M = \abcd{P}{R}{0}{Q}
\]
(again with blocks of size $n \times n$), and the action of $M$ on $A_X$ is given by
\[
M A_X M^{-1} = A_{PXQ^{-1}} = \abcd{\lambda I}{P X Q^{-1}}{0}{\lambda I}.
\]
\end{lemma}

\begin{lemma}\label{lemma:4}
Assume the classification \SC{2} of $3$-element totally symmetric sets in $\GL_2(\C)$. Then up to conjugation, there is a unique $4$-element totally symmetric set $\mathcal A \subset \End(\C^4)$ of the form
\[
A_i = \begin{pmatrix} \lambda I & X_i\\ & \lambda I\end{pmatrix},
\]
where $X_i \in \End(\C^2)$ is invertible. It is given as follows: Define $X_1 = I$ and let $X_2, X_3, X_4$ be given by the noncommutative simplex construction (\Cref{ex:k3n2}) with $\mu$ satisfying $3\mu^2 + 2 \mu + 3=0$ and $\lambda = \mu^{-1}$. Then, taking $A_i$ as
\[
A_i = \begin{pmatrix} \nu I & X_i\\ & \nu I\end{pmatrix},
\]
the set $\mathcal A = \{A_1, A_2, A_3, A_4\}$ is totally symmetric for any choice of $\nu \in \C$.
\end{lemma}
\begin{proof}
Suppose that such a set exists. Via \Cref{lemma:conj}, we may conjugate $\mathcal A$ so as to set $X_1 = I$. We consider the restriction of the realization map to the stabilizer of $A_1$. By \Cref{lemma:conj}, in order to stabilize $A_1$, these must be realized by elements of the form $\abcd{P}{*}{0}{P}$, and so we see that $X_2, X_3, X_4$ forms a $3$-element totally symmetric set in $\GL_2(\C)$. By \SC{2}, we must have
\[
X_2 = \abcd{\lambda}{\frac{\mu-\lambda}{2}}{0}{\mu}, \qquad 
X_3 = \abcd{\mu}{0}{\frac{\mu-\lambda}{2}}{\lambda}, \qquad
X_4 = \abcd{\frac{\lambda+\mu}{2}}{\frac{\lambda-\mu}{2}}{\frac{\lambda-\mu}{2}}{\frac{\lambda+\mu}{2}}
\]
for some parameters $\lambda \ne \mu \in \C^\times$. 

We will see that there is a unique choice of $\mu,\lambda$ for which this is possible. This will be accomplished by analyzing the element $\rho((12))$. By \Cref{lemma:conj}, there are elements $P, Q \in \GL_2(\C)$ for which
\[
\rho((12)) = \abcd{P}{*}{0}{Q}.
\]

\para{Observation 1} {\em $X_2 = PQ^{-1}$ and $P^2 = Q^2$.}

These arise from the fact that $\rho((12))$ must exchange $X_1 = I$ and $X_2$.

\para{Observation 2} {\em $\lambda = \mu^{-1}$ and moreover $\mu^2 \ne 1$.} 

We must have $PX_3Q^{-1} = X_3$, so that $P$ and $Q$ are conjugate. Thus $X_2 = PQ^{-1}$ must have determinant $1$, so that $\lambda = \mu^{-1}$. As $\lambda \ne \mu$, we must also have $\mu^2 \ne 1$.\\

Going forward, we set 
\[
\alpha = \frac{\mu - \lambda}{2} = \frac{\mu - \mu^{-1}}{2}.
\]
Note $\alpha \ne 0$. We observe the following algebraic identities:
\begin{equation}\label{equation:alpha}
    2 \alpha \mu = \mu^2 - 1; \qquad 2 \alpha \mu^{-1} = 1 - \mu^{-2}.
\end{equation}

\para{Observation 3} {\em Without loss of generality, $P$ and $Q$ are given as follows, for some $b \in \C$:}
\[
P = \abcd{1}{b}{2}{-1}, \qquad Q = \abcd{\mu^{-1}}{\alpha + \mu b}{2 \mu^{-1}}{-\mu^{-1}}.
\]
To see this, we write $P = \abcd{a}{b}{c}{d}$. From Observation 1, we see that $PX_2 = Q$ and $X_2 Q = P$. From the first of these, we see
\[
Q = \abcd{a}{b}{c}{d} \abcd{\mu^{-1}}{\alpha}{0}{\mu} = \abcd{\mu^{-1} a}{\alpha a + \mu b}{\mu^{-1} c}{\alpha c + \mu d}.
\]
Inserting this into the second,
\begin{align*}
P = \abcd{a}{b}{c}{d} &= \abcd{\mu^{-1}}{\alpha}{0}{\mu} \abcd{\mu^{-1} a}{\alpha a + \mu b}{\mu^{-1} c}{\alpha c + \mu d}\\
&= \abcd{\mu^{-2}a + \alpha \mu^{-1} c}{\alpha \mu^{-1}a + b + \alpha^2 c + \alpha \mu d}{c}{\alpha \mu c + \mu^2 d}.
\end{align*}
From the upper left, we find 
\[
c = \frac{1- \mu^{-2}}{\alpha \mu^{-1}}a = 2a
\]
(the latter holding by the identity \eqref{equation:alpha}). Likewise the bottom right gives 
\[
c = \frac{1-\mu^2}{\alpha \mu}d = -2d.
\]
Necessarily $c \ne 0$ (otherwise $P$ would not be invertible). As $(\gamma P)X(\gamma Q)^{-1} = P X Q^{-1}$ for any $\gamma \ne 0$, we can scale $P$ so that $c = 2$. Then every entry of $Q$ is visibly of the required form except the bottom right: we must show that $2\alpha - \mu = - \mu^{-1}$, but this is immediate from the definition of $\alpha$. 

To proceed, we analyze the condition $PX_3Q^{-1} = X_3$ (or its equivalent form $PX_3 = X_3Q$) more closely. Multiplying,
\begin{align*}
PX_3 &= \abcd{1}{b}{2}{-1}\abcd{\mu}{0}{\alpha}{\mu^{-1}} &=& \abcd{\mu + \alpha b}{\mu^{-1} b}{2 \mu - \alpha}{-\mu^{-1}},\\
X_3Q &= \abcd{\mu}{0}{\alpha}{\mu^{-1}}\abcd{\mu^{-1}}{\alpha + \mu b}{2 \mu^{-1}}{-\mu^{-1}} &=& \abcd{1}{\alpha \mu + \mu^2 b}{\mu^{-1}(\alpha + 2 \mu^{-1})}{\alpha^2 + \alpha \mu b - \mu^{-2}}.
\end{align*}
From the top left, we find $b = \alpha^{-1}(1-\mu)$. Examining the bottom left, we see that $\mu$ satisfies the constraint
\[
2 \mu - \alpha = \mu^{-1}(\alpha + 2\mu^{-1}).
\]
Thus $\mu$ satisfies the expression
\begin{align*}
& (2\mu^2)(2 \mu - \alpha - \alpha \mu^{-1} - 2 \mu^{-2})\\
        &= 4 \mu^3 - 2 \alpha \mu^2  - 2 \alpha \mu -4\\
        &= 4 \mu^3 - \mu(\mu^2 - 1) - (\mu^2 - 1) - 4\\
        &= 3 \mu^3 - \mu^2 + \mu - 3\\
        &= (\mu-1)(3\mu^2 + 2 \mu + 3).
\end{align*}
As $\mu \ne 1$, we conclude that $\mu$ satisfies $3\mu^2 + 2 \mu + 3 =0$.\\

We have found that if $\mathcal A$ exists, it must be of the claimed form. It remains to check that for this choice of $\mu$, the set is indeed totally symmetric. The set $X_2,X_3,X_4$ is totally symmetric in $\End(\C^2)$ under the standard representation; let the realization map be denoted by $\rho(\sigma) = p_\sigma$. Following \Cref{lemma:conjformula}, we see that the assignment
\[
\sigma \mapsto P_\sigma := \begin{pmatrix} p_\sigma & 0\\ 0 &p_\sigma \end{pmatrix}
\]
stabilizes $A_1$ and realizes $\{A_2,A_3,A_4\}$ as a totally symmetric set in $\End(\C^4)$.

It therefore suffices to verify that the element 
\[
\rho((12)) = \abcd{P}{0}{0}{Q}
\]
with $P,Q$ as above, actually does induce the required permutation. In the supplemental Mathematica notebook \cite{mathematica}, we make the following verifications:
\begin{align*}
PX_1Q^{-1} &= X_2\\
PX_2Q^{-1} &= X_1\\
PX_3Q^{-1} &= X_3\\
PX_4Q^{-1} &= X_4,
\end{align*}
which shows that this is the case.
\end{proof}

For the sake of concreteness, here are the matrices comprising $\mathcal A$:
\begin{align*}
    &A_1 = \left(\begin{array}{cc|cc}
    \nu &       & 1     & 0\\
        &\nu    & 0     & 1\\ \hline
        &       & \nu   & \\
        &       &       & \nu
    \end{array} \right)
        & &A_2 = \left(\begin{array}{cc|cc}
    \nu &       & -\mu - \frac{2}{3} & \mu + \frac{1}{3}\\
        &\nu    & 0     & \mu\\ \hline
        &       & \nu   & \\
        &       &       & \nu
    \end{array}\right) \\
        &A_3 = \left(\begin{array}{cc|cc}
    \nu &       & \mu     & 0\\
        &\nu    & \mu+\frac{1}{3}     & -\mu-\frac{2}{3}\\ \hline
        &       & \nu   & \\
        &       &       & \nu
    \end{array}\right)
        & &A_4 = \left( \begin{array}{cc|cc}
    \nu &       & \frac{-1}{3} & -\mu - \frac{1}{3}\\
        &\nu    & -\mu - \frac{1}{3}     & \frac{-1}{3}\\ \hline
        &       & \nu   & \\
        &       &       & \nu
    \end{array}\right).
\end{align*}

\section{Low-dimensional representations of the symmetric group}\label{section:reptheory}
As an application of these methods, we show how they give a quick proof of the fact that the symmetric group $\Sigma_n$ admits no non-abelian representations of dimension $d < n-1$ for $n \ne 4$. The typical way one sees this is to invoke the full machinery of the representation theory of $\Sigma_n$ and deduce it as an indirect corollary of a formula for the dimensions of the irreps. Here, we show that there is a good structural reason for this dimension gap, in light of the fact that $\Sigma_n$ contains large totally symmetric sets.

\begin{proposition}
\label{proposition:sn}
For $n \ne 4$, the symmetric group $\Sigma_n$ admits no non-abelian representations over $\C$ of dimension $d < n-1$.
\end{proposition}
\begin{proof}
Let $\mathcal A \subset \Sigma_n$ be the $n-1$ element noncommutative totally symmetric set given by
\[
\mathcal A = \{(1,i) \mid 2 \le i \le n\}.
\]
By \Cref{remark:hom}, the image of $\mathcal A$ under any representation $\psi: \Sigma_n \to \GL_d(\C)$ is a totally symmetric set in $\GL_d(\C)$, possibly degenerate. As $\mathcal A$ generates $\Sigma_n$, if $\psi(\mathcal A)$ is degenerate, then necessarily $\psi$ factors through $\Sigma_n^{ab} \cong \Z/2\Z$. Thus, to show that a representation $\psi$ is abelian, it suffices to show that $\psi(\mathcal A)$ is degenerate.

Following \Cref{maintheorem:max}, for $n-1 \ne 5$ there is exactly one non-degenerate totally symmetric set of size $n-1$ in $\GL_{n-2}(\C)$, the noncommutative simplex construction $\Sigma^{nc} \mathcal S_{n-2}$. For $n-1 = 5$ there is the additional possibility $\mathcal A(\tilde \Sigma_5)$ of \Cref{s5construction}. We will show that for $n \ne 4$, it is not possible for a representation $\psi$ to realize $\mathcal A$ as $\Sigma^{nc} \mathcal S_{n-2}$, and that $\mathcal A(\tilde \Sigma_5)$ is likewise infeasible.

The noncommutative simplex construction has two distinct eigenvalues $\lambda, \mu$. As the elements of $\mathcal A$ have order $2$, it follows that $\{\lambda, \mu\} = \{1,-1\}$. Tensoring with the sign representation if necessary, we can assume that $\mu = 1$ and $\lambda = -1$. Then \eqref{equation:ncsimplex} specializes to the formula
\begin{equation}
    \label{equation:aiv}
A_i(v) = v - \tfrac{2}{n-2}\alpha_i(v) e_i
\end{equation}
for the elements $A_i = \psi((1,i))$. For convenience, we record
\[
A_i(e_j) = \begin{cases}
-e_i & i = j\\
e_j + \tfrac{2}{n-2}e_i & i \ne j.
\end{cases}
\]

Observe that in $\Sigma_n$,
\[
(12)(13)(12) = (13)(12)(13) = (23).
\]
We will see that the corresponding relation $A_1 A_2 A_1 = A_2 A_1 A_2$ does not hold when $A_i$ is given by \eqref{equation:aiv}.

We have
\begin{align*}
A_1 A_2 (e_1)   &= A_1 (e_1 + \tfrac{2}{n-2}e_2)\\
                &= -e_1 + \tfrac{2}{n-2}(e_2 + \tfrac{2}{n-2}e_1)\\
                &= \left(\tfrac{4}{(n-2)^2} -1 \right)e_1 + \tfrac{2}{n-2} e_2.
\end{align*}
Thus
\[
A_1 A_2 A_1(e_1) = A_1 A_2 (-e_1) = \left(1-\tfrac{4}{(n-2)^2} \right)e_1 - \tfrac{2}{n-2} e_2,
\]
while
\begin{align*}
A_2 A_1 A_2 (e_1)   &= A_2 \left(\left(\tfrac{4}{(n-2)^2} -1 \right)e_1 + \tfrac{2}{n-2} e_2 \right)\\
&= \left(\tfrac{4}{(n-2)^2} -1 \right) \left(e_1 + \tfrac{2}{n-2}e_2 \right) - \tfrac{2}{n-2}e_2.
\end{align*}
The $e_1$-coefficient of $A_1 A_2 A_1 (e_1)$ is $\left(\tfrac{4}{(n-2)^2} -1 \right)$, while the $e_1$-coefficient of $A_2 A_1 A_2(e_1)$ is $\left(1-\tfrac{4}{(n-2)^2} \right)$. It follows that these can be equal only for $n = 4$ as claimed.

It remains to consider the case where $n = 6$ and $\Sigma_6$ is acting on $\C^4$ such that the totally symmetric set $\{(1\,i)\}$ is realized by $\mathcal A(\tilde \Sigma_5)$. As in the case of the simplex construction, the eigenvalues of $\mathcal A(\tilde \Sigma_5)$ must be $\pm 1$. Using the formulas for $\mathcal A(\tilde \Sigma_5)$ of \Cref{example:s5constructionexplicit}, one finds
\[
A_1 = \left(
\begin{array}{cccc}
 1 & 0 & -\frac{4 i}{\sqrt{3}+3 i} & -\frac{i}{\sqrt{6}} \\
 0 & 1 & -\frac{i}{\sqrt{6}} & \frac{4 i}{\sqrt{3}-3 i} \\
 0 & 0 & -1 & 0 \\
 0 & 0 & 0 & -1 \\
\end{array}
\right) \qquad A_2 = \left(
\begin{array}{cccc}
 -1 & 0 & 0 & 0 \\
 0 & -1 & 0 & 0 \\
 \frac{4 i}{\sqrt{3}-3 i} & \frac{i}{\sqrt{6}} & 1 & 0 \\
 \frac{i}{\sqrt{6}} & -\frac{4 i}{\sqrt{3}+3 i} & 0 & 1 \\
\end{array}
\right),
\]
and from this one computes (see the supplemental Mathematica notebook \cite{mathematica})
\[
(A_1A_2)^3 \ne I,
\]
showing that the relation $((12)(13))^3 = id$ holding in $S_6$ does not hold for the elements $A_1,A_2$.
\end{proof}

\appendix

\section{The $\tilde \Sigma_5$ arrangement}\label{appendix:s5}

Here we present the computations underlying our analysis of the $\tilde \Sigma_5$ arrangement. We will refer throughout to the supplemental Mathematica notebook \cite{mathematica}, in which full details of the calculations are presented.   

\subsection{The basic representation} We begin with a discussion of the {\em basic representation} $V_{basic}^5$, a four-dimensional projective representation of $\Sigma_5$. We follow \cite{HH} as a general reference, but we will use different coordinates that are better adapted to our setting, c.f. \Cref{remark:why}.

\begin{definition}[Representation group $\tilde \Sigma_5$]
The {\em representation group} $\tilde \Sigma_5$ is the extension 
\[
1 \to \Z/2\Z \to \tilde \Sigma_5 \to \Sigma_5 \to 1
\]
which admits the presentation on generators $z, t_{1}, t_{2}, t_{3}, t_{4}$ and relations
\begin{align*}
    &z^2 = 1,\ z \mbox{ central}\\
    &t_{i}^2 = z\qquad &(1 \le i \le 4)\\
    &(t_{i}t_{i+1})^3 = z \qquad &(1 \le i \le 3)\\
    &t_i t_j = z t_j t_i \qquad &(\abs{i-j}\ge 2).
\end{align*}
The map $\tilde \Sigma_5 \to \Sigma_5$ has kernel $\pair {z}$ and sends $t_i$ to the transposition $(i\ i+1)$.
\end{definition}

The {\em basic representation} $V_{basic}^5 \cong \C^4$ is the representation of $\tilde \Sigma_5$ given by $t_i \mapsto T_i$:
\begin{align*}
    T_1 &= \begin{pmatrix}
       0& P_{12}\\P_{12} & 0
    \end{pmatrix}, \qquad 
     &T_2 &= \begin{pmatrix}
       P_{23}& -P_{23}\\0 & -P_{23}
    \end{pmatrix}, \\
    T_3 &= \begin{pmatrix}
       P_{34}& 0\\0 & Q_{34}
    \end{pmatrix}, 
    &T_4 &= \begin{pmatrix}
       P_{45}& 0 \\0 & P_{45}
    \end{pmatrix},
\end{align*}
with, for $\zeta = e^{i \pi/3}$,
\begin{align*}
    P_{12} = P_{23} = \begin{pmatrix}
       0 & -1\\1&0
    \end{pmatrix}, \qquad & 
    P_{34} = \begin{pmatrix}
       0 & -\zeta^2 \\ -\zeta & 0
    \end{pmatrix},&\\
    Q_{34} = 
    \begin{pmatrix}
       0 & -\zeta \\ -\zeta^2 & 0
    \end{pmatrix}, \qquad &
    P_{45} = \frac{i}{\sqrt{3}}\begin{pmatrix}
       \sqrt 2 & 1\\ 1 & -\sqrt 2
    \end{pmatrix}.&
\end{align*}
We recall that we are working here in idiosyncratic coordinates; see \cite[Chapter 6]{HH} for the more customary description.

\begin{lemma}
The assignment $t_i \mapsto T_i$ defines a homomorphism $\rho: \tilde \Sigma_5 \to \GL_4(\C)$, i.e. gives rise to a representation of $\tilde \Sigma_5$.
\end{lemma}
\begin{proof}
Direct verification: see the supplemental Mathematica notebook \cite{mathematica}.
\end{proof}

According to \cite[Theorem 6.2]{HH}, $V_{basic}^5$ is an irreducible representation of $\tilde \Sigma_5$, and so the central element $z$ acts by a scalar; one checks that $z$ acts by $-I$. Thus $V_{basic}^5$ is indeed a projective representation of $\Sigma_5$.

\subsection{The $\tilde \Sigma_5$ arrangement}
Here we present the detailed computations for the $\tilde \Sigma_5$ arrangement, as constructed in \Cref{example:s5tilde}.

\begin{example}\label{example:s5tildeexplicit}
Using the explicit description of the action of $\Sigma_5$ on $V_{basic}^5$ as given in \cite[Chapter 6]{HH} and the knowledge that $V_{basic}^5$ decomposes as a pair of non-isomorphic $\Sigma_4$-representations, one can compute the projection operators $\pi_i: V_{basic}^5 \to W_i$ and arrive at an explicit set of representatives for $\mathcal W_{\tilde \Sigma_5}$ as the images of the following $4 \times 2$ matrices (represented as $2\times 2$ blocks):
\[
W_1= \binom{I}{0}, \quad W_2= \binom{0}{I}, \quad W_3= \binom{I}{I}, \quad W_4 = \binom{A_4}{I}, \quad W_5= \binom{A_5}{I},
\]
where
\[
A_4 = \begin{pmatrix}
   \zeta & 0 \\ 0 & \zeta^{-1}
\end{pmatrix},\qquad
A_5 = \left(
\begin{array}{cc}
 \frac{1}{6} \left(3+i \sqrt{3}\right) & \frac{\sqrt 2 i}{\sqrt{3}} \\
 \frac{\sqrt 2 i}{\sqrt{3}} & \frac{1}{6} \left(3-i \sqrt{3}\right) \\
\end{array}
\right)
\]
for $\zeta = e^{i \pi /3} =  (1 + i\sqrt 3)/2$ a primitive sixth root of unity.
\end{example}

\begin{remark}\label{remark:why}
In the ``standard'' coordinates for $V_{basic}^5$ (as in, e.g. \cite{HH}), the coordinates for $W_i$ would be less suitable for the kinds of analysis we perform on them in \Cref{section:abac} and in \Cref{subsection:halfdim} below. It is for this reason that we change coordinates so as to represent $\mathcal W$ as above.
\end{remark}

\subsection{The $\tilde \Sigma_5$ construction}
Here we describe a procedure for computing explicit matrices for the $\tilde \Sigma_5$ construction of \Cref{s5construction}.
\begin{example}\label{example:s5constructionexplicit}
In \Cref{s5construction}, the element $A_i$ in the $\tilde \Sigma_5$ construction was specified by taking eigenspaces $E_{\lambda}^i = (V_{basic}^4)_i$ and $E_\mu^i = (V_{basic}^{4,a})_i$, where the arrangement $\{(V_{basic}^4)_i\}$ is the $\tilde \Sigma_5$ arrangement $\mathcal W$, and the arrangement $\{V_{basic}^{4,a})_i$ is the arrangement formed by the other irreducible subrepresentation of $V_{basic}^5$ as a projective $\Sigma_4$ representation. Thus we must describe this complementary arrangement $\mathcal W^a$. 

An explicit description of $\mathcal W^a$ in our coordinates is rather complicated. We present it in full in the supplemental Mathematica notebook \cite{mathematica}, along with a full description of the matrices in the $\tilde \Sigma_5$ construction. Here, we note that the first space $W_1^a$ complementary to $W_1$ can be given by
\[
W_1^a = \left(
\begin{array}{cc}
 \frac{i}{2 \sqrt{6}} & \frac{2 i}{\sqrt{3}+3 i} \\
 \frac{2}{3+i \sqrt{3}} & \frac{i}{2 \sqrt{6}} \\
 0 & 1 \\
 1 & 0 \\
\end{array}
\right)
\]
Defining $M \in \GL_4(\C)$ by $M = (W_1 \mid W_1^a)$, one can then construct $\mathcal A_{\tilde \Sigma_5}$ via
\[
A_i = T_{\sigma_i} M \begin{pmatrix}
 \lambda I & \\ & \mu I
\end{pmatrix} M^{-1} T_{\sigma_i}^{-1},
\]
where $T_{\sigma_i} \in \tilde \Sigma_5$ is any element such that the associated $\sigma_i \in \Sigma_5$ satisfies $\sigma_i(1) = i$.
\end{example}

\subsection{Half-dimensional arrangements}\label{subsection:halfdim}
Here we continue with the analysis of half-dimensional arrangements begun in \Cref{section:abac}. We recall that the following results were established there:

\para{\Cref{lemma:coords}} {\em 
Let $W_1, \dots, W_k \subset \C^{2d}$ be $d$-planes such that $W_i \oplus W_j = \C^{2d}$ for any pair of distinct indices $i\ne j$. Then there exist coordinates on $\C^{2d}$ in which $\{W_i\}$ admit $\GL_d(\C)$-coset representatives of the following form:
\[
W_1 = \binom{I}{0}, \qquad W_2 = \binom{0}{I}, \qquad W_3 = \binom{I}{I}, \qquad W_i = \binom{A_i}{I}\ (i \ge 4)
\]
with $A_i$ invertible for $i \ge 4$ (each indicated block is of size $d \times d$). Moreover, for any $M \in \GL_d(\C)$, the matrices $A_i$ for $i \ge 4$ can be sent to the conjugates $M A_i M^{-1}$ while keeping the representatives $W_1, W_2, W_3$ fixed.}\\

\para{\Cref{lemma:halfdim}} {\em 
Let $\mathcal W = \{W_1, \dots, W_k\}$ be a maximal totally symmetric arrangement of $d$-planes in $\C^{2d}$ with the property that $\C^{2d} = W_i \oplus W_j$ for any pair of distinct indices $i \ne j$. Then in the coordinates of \Cref{lemma:coords}, the matrices $A_i$ for $i \ge 4$ form a totally symmetric set in $\GL_d(\C)$ of size $k-3$. In particular, letting $M(d)$ denote the maximal cardinality of a (non-commutative) totally symmetric set in $\End(\C^d)$, then $k \le M(d) + 3$.}\\

\para{\Cref{lemma:evals}} {\em Let $\mathcal W = \{W_1, \dots, W_k\}$ be a maximal totally symmetric arrangement of $d$-planes in $\C^{2d}$ with the property that $\C^{2d} = W_i \oplus W_j$ for any pair of distinct indices $i \ne j$, and let $\{A_4, \dots, A_k\} \subset \GL_d(\C)$ denote the associated totally symmetric set of \Cref{lemma:halfdim}. Let $\Lambda = \{\lambda_1, \dots, \lambda_d\}$ denote the set of eigenvalues of any $A_i$, counted with multiplicity. Then $\Lambda$ is invariant under the involutions $\alpha(\lambda) = \lambda^{-1}$ and $\beta(\lambda) = 1 - \lambda$.}\\

We now continue in the specialized setting of $d = 2$.

\begin{lemma}\label{lemma:d2}
Let $\mathcal W = \{W_1, \dots, W_k\}$ be a nondegenerate totally symmetric arrangement of $2$-planes in $\C^4$ with the property that $\C^4 = W_i \oplus W_j$ for any pair of distinct indices $i \ne j$. Assume the classification \SC{2} of totally symmetric sets of $3$ elements in $\GL_2(\C)$. Then $k \le 5$.
\end{lemma}

\begin{proof}
Suppose for the sake of contradiction that $k \ge 6$. We choose coordinates as in \Cref{lemma:coords}
\[
W_1 = \binom{I}{0}, \qquad W_2 = \binom{0}{I}, \qquad W_3 = \binom{I}{I}, \qquad W_i = \binom{A_i}{I}\ (i \ge 4)
\]
with $\{A_4, A_5, A_6\}$ forming a totally symmetric set in $\GL_2(\C)$ via \Cref{lemma:halfdim}. The hypothesis \SC{2} asserts that $\{A_4, A_5, A_6\}$ must arise via the noncommutative simplex construction. Following \Cref{ex:k3n2}, there is a change of coordinates $M: \C^2 \to \C^2$ and eigenvalues $\lambda, \mu \in \C^\times$ under which the conjugates of $A_4,A_5,A_6$ by $M$ are given by
\begin{equation}\label{equation:a456}
MA_4M^{-1}  = \abcd{\lambda}{\frac{\mu-\lambda}{2}}{0}{\mu} \qquad 
MA_5M^{-1}  = \abcd{\mu}{0}{\frac{\mu-\lambda}{2}}{\lambda} \qquad
MA_6M^{-1}  = \abcd{\frac{\lambda+\mu}{2}}{\frac{\lambda-\mu}{2}}{\frac{\lambda-\mu}{2}}{\frac{\lambda+\mu}{2}}.
\end{equation}
Via \Cref{lemma:coords} we can replace $A_4,A_5, A_6$ by these conjugates; for simplicity we will continue to refer to these matrices as $A_4, A_5, A_6$. 

Let $\Lambda = \{\lambda_1,\lambda_2\}$ denote the roots of the characteristic polynomial of $A_i$, counted with multiplicity. \Cref{lemma:evals} is quickly seen to force $\Lambda = \{\zeta, \zeta^{-1}\}$ to be the primitive sixth roots of unity. 

To show that $k \ge 6$ is impossible, we consider an element $P_{36} \in \GL_4(\C)$ realizing the permutation $(36) \in \Sigma_k$. As this fixes $W_1, W_2$, it is of the form
\[
P_{36} = \begin{pmatrix}
    X & \\ & Y
\end{pmatrix}
\]
for some $X, Y \in \GL_2(\C)$. As $P_{36}W_3 = W_6$ as $\GL_2(\C)$-cosets, there is $Z \in \GL_2(\C)$ such that
\[
\binom{X}{Y} = \binom{A_6Z}{Z};
\]
it follows that $Z = Y$ and hence $X = A_6 Y$. Then as $P_{36}W_6 = W_3$,
\[
A_6 Y A_6 = Y,
\]
and hence $Y$ induces a conjugacy between $A_6$ and $A_6^{-1}$. Such $Y$ must exchange the $\zeta$ and $\zeta^{-1}$ eigenspaces of $A_6$, which are spanned by $e_1 + e_2$ and $e_1-e_2$, respectively. It follows that $Y$ is of the form
\begin{equation}\label{Y}
Y = \begin{pmatrix}
c+d&c-d\\d-c&-c-d
\end{pmatrix}
\end{equation}
for elements $c,d\in \C^\times$.

The condition $P_{36} W_4 = W_4$ gives rise to the equation
\[
 \begin{pmatrix}
    A_6 Y & \\ & Y
\end{pmatrix} \begin{pmatrix}
   A_4 \\ I
\end{pmatrix} = \begin{pmatrix}
   A_6 Y A_4 \\ Y
\end{pmatrix} =  \begin{pmatrix}
    A_4Z \\Z
\end{pmatrix}
\]
for some $Z \in \GL_2(\C)$, from which it follows that $Z = Y$ and ultimately that
\[
A_6 Y A_4 = A_4 Y.
\]
Combining the explicit expressions for $A_4, A_6$ of \eqref{equation:a456} (taking $\lambda = \zeta$ and $\mu = \zeta^{-1}$) with the expression for $Y$ of \eqref{Y} shows, after a calculation, that 
\[
0 = A_6 Y A_4 - A_4 Y = 
\begin{pmatrix}
 \tfrac 1 2 \left((1-2 \sqrt 3 i)c +(-2+ \sqrt 3 i) d \right)& \tfrac{-1}{2}\zeta \left(7 c + \zeta d \right)\\
* & *
\end{pmatrix}.
\]
The only solutions to this are $c = d = 0$, contrary to the requirement that $c, d \in \C^\times$. See the supplemental Mathematica notebook \cite{mathematica}.
\end{proof}

\begin{lemma}\label{lemma:d2unique}
In the setting of \Cref{lemma:d2}, if $\mathcal W = \{W_1, \dots, W_5\}$ is of maximal size, then $\mathcal W$ is the $\tilde \Sigma_5$ arrangement.
\end{lemma}
\begin{proof}
Arguing as in \Cref{lemma:d2}, a combination of \Cref{lemma:coords,lemma:halfdim,lemma:evals} show that $A_4, A_5$ form a $2$-element totally symmetric set with eigenvalues $\zeta, \zeta^{-1}$. We can choose coordinates in which 
\[
A_4 = \begin{pmatrix}\zeta &\\&\zeta^{-1}\end{pmatrix}, \qquad A_5 = M A_4 M^{-1}
\]
for some
\[
M = \begin{pmatrix}a&b\\c&d\end{pmatrix} \in \GL_2(\C).
\]
Adjusting $M$ by a scalar matrix, we assume for later convenience that $\det(M) = ad-bc = -3$. We examine a matrix $P_{34} \in \GL_4(\C)$ realizing the permutation $(34) \in \Sigma_5$. Arguing as in \Cref{lemma:d2}, the effects of $P_{34}$ on the elements $W_1,W_2, W_3$ together force $P_{34}$ to be of the form
\[
P_{34} = \begin{pmatrix}
   A_4 Y &\\ & Y
\end{pmatrix}
\]
for some $Y \in \GL_2(\C)$. From the condition $P_{34} W_4 = W_3$, we see that $A_4 Y A_4 = Y$, and so, as in \Cref{lemma:d2}, $Y$ conjugates $A_4$ to $A_4^{-1}$ and so must be of the form
\[
Y = \begin{pmatrix}
   0 & q\\p&0
\end{pmatrix}
\]
for $p, q \in \C^\times$. 

Arguing as in \Cref{lemma:d2}, the condition $P_{34}W_5 = W_5$ leads to the equality
\[
A_4 Y A_5 = A_5Y.
\]
Computing explicitly (see the supplemental Mathematica notebook \cite{mathematica}),
\[
0 = A_4 Y A_5 - A_5 Y=\left(
\begin{array}{cc}
 * & \frac{i \left(\sqrt{3}+i\right) q (a d+2 b c)}{6} \\
* & *
\end{array}
\right)
\]
and hence $ad + 2bc = 0$. As we fixed $ad-bc = -3$, it follows that $bc = 1$ and so $ad = -2$. Inserting these into $M$ shows that $A_5 = M A_5 M^{-1}$ has the form
\[
A_5 = \left(
\begin{array}{cc}
 \frac{1}{6} \left(3+i \sqrt{3}\right) & \frac{i a b}{\sqrt{3}} \\
 \frac{2 i}{\sqrt{3} a b} & \frac{1}{6} \left(3-i \sqrt{3}\right) \\
\end{array}
\right).
\]

To finish the argument, we make one final change of coordinates. The matrix $A_4$ is centralized by any diagonal matrix. Conjugating $A_5$ by $\abcd{1}{0}{0}{ab/\sqrt 2}$ sends $A_5$ to
\[
A_5' = \left(
\begin{array}{cc}
 \frac{1}{6} \left(3+i \sqrt{3}\right) & \frac{\sqrt 2 i}{\sqrt{3}} \\
 \frac{\sqrt 2 i}{\sqrt{3}} & \frac{1}{6} \left(3-i \sqrt{3}\right) \\
\end{array}
\right),
\]
which visibly realizes $\mathcal W$ in the form given in \Cref{example:s5tildeexplicit}.
\end{proof}

\para{The stabilizer of the $\mathbf{\tilde \Sigma_5}$ arrangement} 
 As a counterpart to the discussion in \Cref{lemma:stabstd}, we show here that the $\tilde \Sigma_5$ arrangement also has minimal stabilizer.
 
 \begin{lemma}\label{lemma:stabs5}
$\Stab(\mathcal W_{\tilde \Sigma_5}) = \C^\times I$.
 \end{lemma}
 \begin{proof}
 We recall the coordinate representation of $\mathcal W_{\tilde \Sigma_5}$ of \Cref{example:s5tildeexplicit}:
 \[
W_1= \binom{I}{0}, \quad W_2= \binom{0}{I}, \quad W_3= \binom{I}{I}, \quad W_4 = \binom{A_4}{I}, \quad W_5= \binom{A_5}{I},
\]
where
\[
A_4 = \begin{pmatrix}
   \zeta & 0 \\ 0 & \zeta^{-1}
\end{pmatrix},\qquad
A_5 = \left(
\begin{array}{cc}
 \frac{1}{6} \left(3+i \sqrt{3}\right) & \frac{\sqrt 2 i}{\sqrt{3}} \\
 \frac{\sqrt 2 i}{\sqrt{3}} & \frac{1}{6} \left(3-i \sqrt{3}\right) \\
\end{array}
\right)
\]
 If $M \in \GL_4(\C)$ fixes the subspaces $W_1, W_2, W_3$, it must be of the form
 \[
M = \begin{pmatrix}
   X & \\ & X
\end{pmatrix}
 \]
 for $X \in \GL_2(\C)$. If $M$ moreover fixes $W_4$ and $W_5$, then $X$ must commute with both $A_4$ and $A_5$. As the eigenspaces for $A_4$ and $A_5$ are distinct, it follows that $X$ must be scalar, and hence $M$ is scalar as well.
 \end{proof}
 
As an immediate corollary, we obtain the counterpart to \Cref{lemma:mustbestd}.
\begin{corollary}\label{corollary:reals5}
There is a unique projective class $\overline \rho: \Sigma_5 \to \GL_4(\C) \to \PGL_4(\C)$ of realization map for $\mathcal W_{\tilde \Sigma_5}$, realized by the basic projective representation of $\Sigma_5$.
\end{corollary}

 \subsection{Non-existence of certain totally symmetric sets}
 Here we show that the $\tilde \Sigma_5$ arrangement {\em cannot} be used as the arrangement of eigenspaces for a totally symmetric set of matrices with nontrivial Jordan blocks.
\begin{lemma}
\label{lemma:halfdim1espc}
Let $\mathcal W_{\tilde \Sigma_5} = \{W_1, \dots, W_5\}$ denote the $\tilde \Sigma_5$ arrangement. There is no totally symmetric set $\mathcal A = \{A_1,\dots,A_5\} \subset \GL_{4}(\C)$ for which $E_{\lambda,1}(A_i) = W_i$ and $E_{\lambda,2}(A_i) = \C^{4}$.
\end{lemma}
\begin{proof}
Supposing the existence of such, in our usual coordinates $A_1$ would have a representative of the form
\[
A_1 = \begin{pmatrix}
 \lambda I & X\\ 0 & \lambda I
\end{pmatrix}
\]
for some $X \in \GL_2(\C)$. According to \Cref{corollary:reals5}, the totally symmetric set $\mathcal A$ would have to have realization map given by the basic representation $\rho$ of $\tilde \Sigma_5$. In particular, the elements $T_i$ for $i \ge 2$, which represent the transpositions $(i\,i+1)$, must commute with $A_1$. Computing $T_i A_1 T_i^{-1}$ for $i = 2,3,4$ respectively and extracting the top right block, these give the following equations:
\[
X = -P_{23}XP_{23}^{-1}, \qquad X = P_{34} X Q_{34}^{-1}, \qquad X = P_{45}X P_{45}^{-1}.
\]
The first of these shows that $X$ is conjugate to $-X$ and hence the eigenvalues of $X$ must be $\pm \mu$ for some $\mu \in \C^\times$. 

Turning to the third equation, we observe that
\[
P_{45}  = \frac{i}{\sqrt{3}}\begin{pmatrix}
       \sqrt 2 & 1\\ 1 & -\sqrt 2
    \end{pmatrix}
    \]
has two distinct eigenspaces and trace zero. As $X$ commutes with $P_{45}$, it must preserve each eigenspace, and so have the same eigenspaces. As $X$ also has trace zero, it follows that $X = c P_{45}$ for some nonzero constant $c$.

Finally the second equation now implies that $P_{45} = P_{34} P_{45} Q_{34}^{-1}$, but this is directly computed to be false; as seen in the supplemental Mathematica notebook \cite{mathematica},
\[
P_{45} Q_{34} - P_{34}Q_{34} = \left(
\begin{array}{cc}
 0 & \sqrt{2} \\
 -\sqrt{2} & 0 \\
\end{array}
\right).
\]
\end{proof}

\bibliographystyle{alpha}
\bibliography{bib}

\begin{thebibliography}{CKLP20}

\bibitem[CK20]{Caplinger-Kordek}
Noah Caplinger and Kevin Kordek.
\newblock Small quotients of braid groups.
\newblock {P}reprint, https://arxiv.org/pdf/2009.10139.pdf, 2020.

\bibitem[CKLP20]{Finite_Quotients}
Alice Chudnovsky, Kevin Kordek, Qiao Li, and Caleb Partin.
\newblock Finite quotients of braid groups.
\newblock {\em Geom. Dedicata}, 207:409--416, 2020.

\bibitem[CM20]{Chen-Mukherjea}
Lei Chen and Aru Mukherjea.
\newblock From braid groups to mapping class groups.
\newblock {P}reprint, https://arxiv.org/pdf/2011.13020.pdf, 2020.

\bibitem[CS22]{mathematica}
N.~Caplinger and N.~Salter.
\newblock Supplemental {M}athematica notebook.
\newblock \url{https://github.com/nick-salter/TSS.git}, 2022.

\bibitem[HH92]{HH}
P.~N. Hoffman and J.~F. Humphreys.
\newblock {\em Projective representations of the symmetric groups}.
\newblock Oxford Mathematical Monographs. The Clarendon Press, Oxford
  University Press, New York, 1992.
\newblock $Q$-functions and shifted tableaux, Oxford Science Publications.

\bibitem[KLP21]{UpperBoundsTSS}
Kevin Kordek, Qiao Li, and Caleb Partin.
\newblock Upper bounds for totally symmetric sets.
\newblock {P}reprint, https://arxiv.org/pdf/2102.06270.pdf, 2021.

\bibitem[KM19]{Kordek-Margalit}
Kevin Kordek and Dan Margalit.
\newblock Homomorphisms of commutator subgroups of braid groups.
\newblock {P}reprint, https://arxiv.org/pdf/1910.06941.pdf, 2019.

\bibitem[SV20]{Nancy-Yvon}
Nancy Scherich and Yvon Verberne.
\newblock Finite image homomorphisms of the braid group and its
  generalizations.
\newblock {P}reprint, https://arxiv.org/pdf/2012.01378.pdf, 2020.

\end{thebibliography}

\end{document}